\documentclass[letterpaper,11pt]{amsart}

\usepackage[margin=1.2in]{geometry}
\usepackage{amsmath,amsthm,amssymb}
\usepackage{xspace,xcolor}
\usepackage[breaklinks,colorlinks,citecolor=teal,linkcolor=teal,urlcolor=teal,pagebackref,hyperindex]{hyperref}
\usepackage[alphabetic]{amsrefs}
\usepackage{comment}
\usepackage[all]{xy}
\usepackage{tikz-cd}
\usepackage{color}

\theoremstyle{plain}
\newtheorem{theo}{Theorem}[section]

\newtheorem{lemm}[theo]{Lemma}
\newtheorem{prop}[theo]{Proposition}
\newtheorem{coro}[theo]{Corollary}

\theoremstyle{definition}
\newtheorem{defi}[theo]{Definition}

\theoremstyle{remark}
\newtheorem{rema}[theo]{Remark}

\newcommand{\eps}{\epsilon}

\newcommand{\bfR}{\mathbf{R}}

\newcommand{\bb}[1]{\mathbb{#1}}
\newcommand{\HH}{\mathbb{H}}
\newcommand{\PP}{\mathbb{P}}
\newcommand{\QQ}{\mathbb{Q}}

\newcommand{\RR}{\mathbb{R}}
\newcommand{\CC}{\mathbb{C}}
\newcommand{\DD}{\mathbb{D}}

\newcommand{\ZZ}{\mathbb{Z}}

\newcommand{\cA}{\mathcal{A}}

\newcommand{\cC}{\mathcal{C}}
\newcommand{\cD}{\mathcal{D}}

\newcommand{\cH}{\mathcal{H}}

\newcommand{\cK}{\mathcal{K}}
\newcommand{\cL}{\mathcal{L}}
\newcommand{\cM}{\mathcal{M}}
\newcommand{\cN}{\mathcal{N}}
\newcommand{\cO}{\mathcal{O}}
\newcommand{\cP}{\mathcal{P}}

\newcommand{\uind}[1]{^{(#1)}}
\newcommand{\lind}[1]{_{(#1)}}

\newcommand{\lbind}[1]{_{[#1]}}
\newcommand{\lsta}{_{*}}

\newcommand{\sta}{^{*}}

\newcommand{\norm}[1]{\|#1\|}

\DeclareMathOperator{\id}{id}

\DeclareMathOperator{\gr}{gr}

\DeclareMathOperator{\SProj}{\bf{Proj}}

\DeclareMathOperator{\coh}{coh}

\DeclareMathOperator{\supp}{supp}

\newcommand{\de}{\partial}
\newcommand{\db}{\overline{\partial}}

\newcommand{\zbar}{\overline{z}}

\newcommand{\lap}{\Delta}

\DeclareMathOperator{\Dom}{Dom}

\DeclareMathOperator{\GL}{GL}

\DeclareMathOperator{\DR}{DR}

\DeclareMathOperator{\MHM}{MHM}
\DeclareMathOperator{\HM}{HM}

\newcommand{\bfu}{\mathbf{u}}

\DeclareMathOperator{\Poin}{car\acute{e}}

\usepackage{soul}

\begin{document}
	\thanks{The author was partially supported by NSF grant DMS-2001132.}
	
	
	\author{Hyunsuk Kim}
	
	\address{Department of Mathematics, University of Michigan, 530 Church Street, Ann Arbor, MI 48109, USA}
	
	\email{kimhysuk@umich.edu}

	\begin{abstract} 
		We give an analytic proof of the Saito vanishing theorem using $L^{2}$-methods, by going back to the original idea for the proof of the Kodaira vanishing theorem.
	\end{abstract} 
	
	\title[$L^{2}$-approach to the Saito vanishing thoerem]{$L^{2}$-approach to the Saito vanishing thoerem}

	\maketitle
	
	\setcounter{tocdepth}{1}
	\tableofcontents 
	
	\section{Introduction}
    Along the way of developing the theory of mixed Hodge modules in \cite{saito1988modules,saito1990mixed}, Saito proved a vanishing theorem \cite{saito1990mixed}*{2.g} that can be understood as a far-reaching generalization of the Kodaira vanishing theorem, which plays a crucial role in algebraic geometry in characteristic 0. Almost all of the robust vanishing theorems involving ample line bundles are special cases of Saito vanishing. For example, if we apply the theorem to the trivial variation of Hodge structures on a smooth projective variety, then we recover the classical Kodaira--Akizuki--Nakano vanishing. Saito's vanishing  theorem concerns polarizable mixed Hodge modules for which key examples are provided by polarizable variations of Hodge structures. The theorem says the following:

    \begin{theo} \cite{saito1990mixed}*{2.g}
		Let $X$ be a complex projective variety and let $\cM \in \MHM(X)$ be a polarizable Hodge module on $X$. For an ample line bundle $\cL$ on $X$, we have
		$$ \HH^{l}\big( X, \gr_{p} \DR_{X}(\cM) \otimes \cL \big) = 0 \qquad \textnormal{for all } l > 0, p \in \ZZ.$$
	\end{theo}

    In this article, we provide a proof of Saito's vanishing theorem for mixed Hodge modules using analytic methods. We also mention that the proof works in a more general setup, by allowing complex coefficients, since our analytic method has nothing to do with the underlying $\QQ$-structure.\footnote{The author would like to thank Rui-jie Yang for pointing this out.} This is the setup of \textit{complex Hodge modules}, being developed in the mixed Hodge module project by Sabbah--Schnell \cite{SSMHMproject}. This is a generalized notion of Saito's Hodge modules which removes $\QQ$-perverse sheaves in the picture and replaces with certain distribution-valued sesquilinear pairing on $\cD$-modules.

    It is not a big surprise that the known proofs of the Saito vanishing theorem have some correspondence with the known proofs of the Kodaira vanishing theorem. One of the key ingredients of Saito's original approach in \cite{saito1990mixed}*{2.g} is the Artin--Grothendieck vanishing theorem, which is a perverse sheaf version of Andreotti--Frankel's result. This is a key input of Ramanujam's approach of Kodaira vanishing \cite{ramanujam1972remarks}, hence we view Saito's method as a generalization of his. There is another proof of Saito's vanishing theorem in \cite{schnell2016saito} (or  \cite{Schnell-Yang:2023higher}*{Theorem 4.5} for complex coefficients). This approach can be viewed as a cousin of Esnault--Viehweg's \cite{esnault1985logarithmic} which uses branched covering techniques and the degeneration of the Hodge to de Rham spectral sequence. We also mention a partial progress towards Saito vanishing using positive characteristic methods in \cite{arapura2019kodaira, arapura2019vanishing} which are in spirit, similar to the approach of Deligne--Illusie \cite{deligne1987relevements}. We mention that our method goes back to the very first proof of the Kodaira vanishing \cite{kodaira1953differential}, exploiting the positivity of the curvature of the ample line bundle.
	
	We describe our general strategy for the proof of Saito's vanishing theorem. First, we improve the result of \cite{deng2022vanishing} and prove a vanishing theorem (Theorem \ref{theo-Main-ample}) for a certain logarithmic de Rham complex associated to a variation of Hodge structures defined on the complement of a simple normal crossing (SNC) divisor. After that, we reduce Saito's vanishing theorem to Theorem \ref{theo-Main-ample} using the structure theory for complex Hodge modules. We summarize the three main advantages of this argument. First, the proof of Theorem \ref{theo-Main-ample} is established through analytic methods and does not require any results on Hodge modules or $\cD$-modules. Specifically, we only use the curvature formula for Hodge bundles and a careful analysis near the SNC divisor by studying the degeneration of Hodge structures. Second, Theorem \ref{theo-Main-ample} is slightly stronger; that is, it does not directly follow from Saito's vanishing theorem. In fact, analytic arguments naturally allow us to perturb a line bundle by an SNC divisor, leading to a Kawamata--Viehweg type vanishing theorem. Lastly, the general case of Saito vanishing almost immediately follows via the structure theory of Hodge modules and by resolving singularities.
	
	\subsection{Vanishing theorem for logarithmic de Rham complexes} We give a vanishing theorem for logarithmic de Rham complexes for complex variations of Hodge structures defined on the complement of an SNC divisor. This is a key ingredient towards the proof of Saito's vanishing theorem. We consider a compact Kähler manifold $\overline{X}$ and an SNC divisor $D = \sum_{i=1}^{\nu} D_{i}$ on $\overline{X}$. For a complex polarizable variation of Hodge structures $E$ on the open locus $X = \overline{X}\setminus D$, we consider the Deligne extension $E_{\alpha}$, for $\alpha \in \RR^{\nu}$, whose eigenvalues of the residue along $D_{i}$ lie inside the interval $[-\alpha_{i}, -\alpha_{i}+1)$. By the nilpotent orbit theorem, the graded pieces $E^{p,q}$ also extend to vector bundles $E_{\alpha}^{p,q}$ on $\overline{X}$. The graded piece of the logarithmic de Rham complex $\gr^{p} \DR_{(\overline{X}, D)} (E_{\alpha})$ is defined as follows:
	$$ \Big[ E_{\alpha}^{p,q} \to \Omega_{\overline{X}}^{1}(\log D) \otimes E_{\alpha}^{p-1, q+1} \to \cdots \to \Omega_{\overline{X}}^{n}(\log D) \otimes E_{\alpha}^{p-n, q+n} \Big] [n].$$
    Note that this complex lives in cohomological degree between $-n$ and $0$. In this situation, we have the following vanishing theorem:
	
	\begin{theo}\label{theo-Main}
		With the above notation, let $\cL$ be a line bundle on $\overline{X}$. Assume that $\cL + \sum_{i=1}^{\nu} \alpha_{i} D_{i}$ has a smooth hermitian metric whose curvature is positive semi-definite, and has at least $n-t$ positive eigenvalues at each point $x \in \overline{X}$. If $B$ is a nef line bundle on $\overline{X}$, then we have the following vanishing:
		$$ \HH^{l}\left( \overline{X} , \gr^{p} \Big(\DR_{(\overline{X}, D)}(E_{\alpha}) \Big) \otimes \cL \otimes B \right) = 0 \qquad \textnormal{ for all } l > t, p \in \ZZ.$$
	\end{theo}
	
	The curvature condition on $\cL + \sum_{i=1}^{\nu} \alpha_{i} D_{i}$ seems somewhat complicated, but a typical situation to get this condition is when the line bundle is the pull-back of an ample one by a smooth proper morphism $f \colon X \to A$, where $\dim A = n-t$. This is a similar generalization to that of Kodaira vanishing given by Girbau \cite{girbau1976theoreme}. For convenience, we also give a simpler statement which is algebraic and is closer to that of the classical Saito vanishing theorem.
	
	\begin{theo} \label{theo-Main-ample}
		With the same notation as in Theorem \ref{theo-Main}, if $\cL + \sum_{i=1}^{\nu} \alpha_{i} D_{i}$ is ample, then
		$$\HH^{l}\left( \overline{X} , \gr^{p} \Big(\DR_{(\overline{X}, D)}(E_{\alpha}) \Big) \otimes \cL \right) = 0 \qquad \textnormal{ for all } l > 0, p \in \ZZ.$$
	\end{theo}
	{\it Proof.} Let $B$ be the trivial bundle. The ampleness condition corresponds to the case $t = 0$ in Theorem \ref{theo-Main}. \hfill{$\square$}
	
	As an immediate corollary, we obtain the vanishing theorem in \cite{huang2016logarithmic,AMPW}, which is a generalization of log-Nakano vanishing where we allow an extra perturbation by an SNC divisor, as a direct corollary of Theorem \ref{theo-Main-ample}.
	
	\begin{coro} \label{coro-classical}
		Let $X$ be a projective complex manifold of dimension $n$ and let $\cL$ be a line bundle and $D = \sum_{i=1}^{\nu}D_{i}$ be a reduced SNC divisor on $X$. If $\cL + \sum_{i=1}^{\nu}\alpha_{i} D_{i}$ is ample for some $0 \leq \alpha_{i} \leq 1$, then
		\begin{align*}
			H^{q}(X, \Omega_{X}^{p}(\log D) \otimes \cL) = 0 & \qquad \textnormal{for } p+q > n. 
		\end{align*}
	\end{coro}
	{\it Proof.} Apply Theorem \ref{theo-Main-ample} for the trivial variation of Hodge structures on the complement $X \setminus D$. \hfill{$\square$}
	
	\subsection{Comparison with previous results}
	We now compare our result with several related previous results. The most similar results and ideas can be found in \cite{deng2022vanishing}. The main theorem of that paper discusses vanishing theorems for tame harmonic bundles with nilpotent residues. The key input of the proof is the bound for the Higgs field near the SNC divisor. Then one can build a resolution of the logarithmic de Rham complex by fine sheaves, which is essentially the Dolbeault resolution but allowing measurable coefficients with certain $L^{2}$-conditions. The chain map is given by $\db + \theta$ and the authors use appropriate $L^{2}$-estimates to solve the $(\db + \theta)$-equation.
	
	Theorem \ref{theo-Main} generalizes \cite{deng2022vanishing} in the case of a variation of Hodge structures. The main difference is that we get vanishing for a family of sheaves parameterized by $\RR^{\nu}$, and assuming some positivity for the line bundle with an $\RR$-twist $\cL + \sum_{i=1}^{\nu} \alpha_{i} D_{i}$. This can be done by a careful construction of the Kähler metric on the open locus and by relating certain sheaves given by $L^{2}$-conditions and the Deligne extension. The result in \cite{deng2022vanishing} can be viewed as the case when $\alpha = 0 \in \RR^{\nu}$. However, the statement in \cite{deng2022vanishing} is stronger since they work with tame harmonic bundles with nilpotent residues, a more general notion than variations of Hodge structures. However, we would like to stress that allowing a perturbation of $\alpha$ is extremely important towards the proof of the Saito vanishing since it allows us to resolve singularities and use the standard perturbation trick to get positivity of the line bundle on the resolution.

    We also compare our results with those in \cite{PopaVanishing}, \cite{SuhSaito}, and \cite{Wu22vanishing} which are Kawamata--Viehweg type generalizations of Saito vanishing. The results in \cite{PopaVanishing}*{\S 11} are somewhat orthogonal to ours since the author considers the situation when the line bundle $\cL$ is big and nef, while we discuss the issue of perturbing by an extra SNC divisor. The result in \cite{SuhSaito} is closely related to ours since it considers vanishing of a logarithmic de Rham complex when we have a perturbation by an SNC divisor. The main result \cite{SuhSaito}*{Theorem 1} can be deduced from Theorem \ref{theo-Main-ample} of ours. The author shows the assertion when $\alpha = -(\eps, \ldots, \eps)$ (with appropriate duality applied), under the assumption that $\cL$ is big and nef, and $\cL + \sum_{i=1}^{\nu} \alpha_{i} D_{i}$ is ample for all $0 < -\alpha_{i} \ll 1$. The assumption of our result is weaker in the sense that we only need a single $\alpha$ with $E_{\alpha} = E_{-\eps}$ and $\cL + \sum_{i=1}^{\nu} \alpha_{i} D_{i}$ is ample and do not impose any positivity condition on $\cL$ itself. The vanishing result in \cite{Wu22vanishing}*{Theorem 1.5} generalizes \cite{PopaVanishing} and \cite{SuhSaito} simultaneously. The theorem shows a vanishing result for the lowest graded piece, under a weaker positivity condition, i.e., assuming that $\cL + \sum \alpha_{i} D_{i}$ is big and nef. In general, vanishings for the higher graded pieces are no longer true for big and nef line bundles. We also mention a related result for twistor $\cD$-modules by \cite{wei2022kodaira}.

    {\bf Acknowledgements.} The author would like to thank his advisor Mircea Musta\c{t}\u{a} for generous support and guidance, Qianyu Chen and Brad Dirks for helpful suggestions and answering questions on Hodge modules. The author would also like to thank Christian Schnell on helpful comments and conversations, and Ya Deng for comments and for pointing out some misunderstandings on harmonic bundles. Lastly, we thank Ruijie Yang for pointing out that the vanishing theorem would work for polarizable complex Hodge modules.
	
	\section{Preliminaries}
	\subsection{Complex Polarized Variation of Hodge Structures} \label{sec-CPVHS}
	We begin by reviewing some basic definitions and notation related to complex variations of Hodge structures. Let $E$ be a smooth vector bundle on a complex manifold $X$. We denote by $\cA_{X}^{r,s}(E)$ the sheaf of smooth differentiable $(r, s)$-forms with values in $E$. A complex variation of Hodge structures on $X$ with weight $k$ is a smooth vector bundle $E$ on $X$ with a decomposition
	$$ E = \bigoplus_{p + q = k} E^{p,q}$$
	by smooth vector bundles and a flat connection $\nabla \colon E \to \cA_{X}^{1} (E)$ that maps each $\cA_{X}(E^{p,q})$ into
	$$ \cA_{X}^{1,0} (E^{p,q}) \oplus \cA_{X}^{1,0} (E^{p-1, q+1}) \oplus \cA_{X}^{0,1} (E^{p,q}) \oplus \cA_{X}^{0,1} (E^{p+1, q-1}).$$
	We define the filtration on $E$ as $F^{p}E = \bigoplus_{p' \geq p} E^{p', k-p'}$. The four components of the connection extend naturally to four operators
	\begin{align*}
		&\nabla^{1,0} \colon \cA_{X}^{r,s} (E^{p,q}) \to \cA_{X}^{r+1, s}(E^{p,q}) &\theta : \cA_{X}^{r,s} (E^{p,q}) \to \cA_{X}^{r+1, s} (E^{p-1,q+1} ) \\
		& \db \colon \cA_{X}^{r,s}(E^{p,q}) \to \cA_{X}^{r, s+1} (E^{p,q}) &  \varphi \colon \cA_{X}^{r, s} (E^{p,q}) \to \cA_{X}^{r, s+1} (E^{p+1, q-1}).
	\end{align*}
	A polarization on $E$ is a sesquilinear pairing $Q \colon E \otimes \overline{E} \to \cA_{X}$ such that 
	\begin{itemize} 
		\item $Q$ is compatible with $\nabla$ in the sense that
		$$ \nabla Q(u, v) = Q(\nabla u , v) + Q(u, \nabla v)$$
		for all smooth sections $u, v$ of $E$, and
		\item The summands $E^{p,q}$ of the decomposition are mutually orthogonal to each other, and
		\item $h(v, w) = \sum_{p+q = k} (-1)^{q} Q(v^{p,q} , w^{p,q})$ is positive definite.
	\end{itemize}
	We denote the hermitian metric on $E$ by $h_{E}$. We will sometimes omit the subscript if there is no possibility for confusion. Note that the vector bundles $E$, $F^{p}E$, and $E^{p,q}$ have holomorphic structures by the $(0,1)$-part of the connection, but the decomposition $E = \bigoplus E^{p,q}$ does not preserve the holomorphic structures. However, $F^{\bullet} E$ are holomorphic subbundles of $E$. This is essentially because of the presence of the Higgs field, which is the operator $\theta$. We say that $(E, \nabla, F^{\bullet}, Q)$ is a complex polarized variation of Hodge structures if the above properties are satisfied. However, we will simply say that $E$ is a complex polarized variation of Hodge structures, without specifying the connection, filtration, and the hermitian paring if there is no possibility for confusion. By the flatness of connection $\nabla^{2} = 0$ and by type analysis, we have the following identities.
	\begin{lemm} \label{lemm-flat-identity}
		The operators $\nabla^{1,0}$, $\db$, $\theta$, $\varphi$ satisfy the following identities:
		\begin{align*}
			& (\nabla^{1,0})^{2}  = \theta^{2}  = \db^{2} = \varphi^{2} = 0 \\
			& \nabla^{1,0} \theta + \theta \nabla^{1,0} = \db \varphi + \varphi\db = 0 \\
			& \db \theta + \theta \db = \nabla^{1,0} \varphi + \varphi \nabla^{1,0} = 0 \\
			& \nabla^{1,0} \db + \db \nabla^{1,0} + \theta \varphi + \varphi \theta = 0.
		\end{align*} 
	\end{lemm}
	
	Here are some more basic facts that will be used later. The proofs can be found in a self-contained manner in \cite{sabbah2022degenerating}*{\S 47 - \S 50}.
	\begin{lemm} \label{lemm-curvature}
		\begin{enumerate} \itemsep -2pt
			\item For all smooth sections $u , v \in E$, we have
			$$ h_{E}(\theta u, v) =h_{E}(u, \varphi v).$$
			\item For each $E^{p,q}$, the connection $\nabla^{1,0} + \db$ is the metric connection with respect to $h$.
			\item The curvature of the hermitian bundle $E^{p,q}$ is equal to $- (\theta \varphi + \varphi \theta)$. \hfill{$\square$}
		\end{enumerate}
	\end{lemm}
	
	\begin{rema}
		We mention some notational differences from \cite{sabbah2022degenerating}. This is mainly because we wanted to reserve the notation $\theta\sta$ for the adjoint of the operator of $\theta$ when $X$ has a Kähler metric.
	\end{rema}
	
	\subsection{Bound for the Norm of Higgs field} \label{sec-Higgs-field-bound}
	We introduce some useful terminology to describe the analysis near a simple normal crossing divisor.
	\begin{defi}[Admissible Coordinates]
		Let $X$ be a complex manifold and let $D$ be a reduced simple normal crossing divisor on $X$. Let $x$ be a point of $X$. An admissible coordinate centered at $x$ is a tuple $(\Omega; z_{1}, \ldots, z_{n})$ where
		\begin{itemize}
			\item $\Omega$ is an open subset of $X$ containing $x$.
			\item $(z_{1},\ldots, z_{n})$ is a coordinate system on $\Omega$ centered at $x$, which gives a holomorphic isomorphism of $\Omega$ with $\Delta^{n} = \{ (\zeta_{1},\ldots, \zeta_{n}) \in \CC^{n} : |\zeta_{j}| < 1\}$.
			\item $D \cap \Omega$ is given by the equation $z_{1}\cdots z_{l} =0$ for some $l \leq n$.
		\end{itemize}
		We shall write $\Omega\sta = \Omega\setminus D$.
	\end{defi}
	We define a Poincaré type metric $\omega_{\Poin}$ on $(\Delta\sta)^{l} \times \Delta^{n-l}$ by
	$$ \omega_{\Poin} = i \sum_{j=1}^{l} \frac{dz_{j} \wedge d\zbar_{j}}{|z_{j}|^{2} (-\log |z_{j}|^{2})^{2}} + \sum_{k=l+1}^{n} i dz_{k} \wedge d\zbar_{k}. $$
    Here, $\Delta = \{ z \in \CC : |z|  <1 \}$ is the unit disk and $\Delta\sta = \Delta \setminus \{0 \}$ is the punctured disk. 
	
	For a variation of Hodge structures $E$ on $(\Delta\sta)^{l} \times \Delta^{n-l}$, we have the following crucial norm estimate for its Higgs field $\theta$ proved in \cite{simpson1990harmonic} for the one-dimensional case and \cite{mochizuki2002asymptotic} for the general case. Its proof follows from a clever use of Ahlfors lemma.
	
	\begin{theo}[Mochizuki] \label{theo-HiggsBound}
		If $E$ is a variation of Hodge structures on $X = \overline{X} \setminus D$, then every point $x \in D$ has an admissible coordinate $(\Omega ; z_{1},\ldots, z_{n})$ centered at $x$ such that
		$$ | \theta|_{h, \omega_{\Poin}}^{2} \leq C$$
		holds on $\Omega\sta$ for some constant $C > 0$. \hfill{$\square$}
	\end{theo}

	\begin{rema}
		We will use this constant $C$ such that $| \theta u|_{h, \omega_{\Poin}}^{2} \leq C |u|_{h, \omega_{\Poin}}^{2}$ for all $(r, s)$-forms valued in $E$. The original statement in \cite{mochizuki2002asymptotic} says that we can get a constant $C_{0}$ such that if we write $\theta = \sum \theta_{j} dz_{j}$, then 
		\begin{align*}
			& | \theta_{i}|_{h, \omega}^{2} \leq \frac{C_{0}}{|z_{i}|^{2} (-\log |z_{i}|^{2})^{2} } \quad \text{for } 1 \leq i \leq l, \quad \text{and} \\
			& | \theta_{i}|_{h, \omega}^{2} \leq C_{0} \quad \text{for } l+1 \leq i \leq n,
		\end{align*}
		where $\omega$ is the usual Euclidean metric. We can get the bound $|\theta u|_{h, \omega_{\Poin}}^{2} \leq C |u|_{h, \omega_{\Poin}}^{2}$ for all $(r, s)$-forms $u$ by multiplying $C_{0}$ by an extra constant only depending on the dimension $n$.
	\end{rema}
	
	\subsection{Prolongation via Norm Growth and the Nilpotent Orbit Theorem} \label{sec-prolongation}
    We address the key consequences of the nilpotent orbit theorem (\cite{schmid1973variation, cattani1986degeneration} for integral variations of Hodge structures and \cite{sabbah2022degenerating,deng2022nilpotent} for complex variations) and the asymptotics of the Hodge norms. In particular, we relate Deligne's extension and the prolongation bundles via the growth of Hodge norms. We briefly summarize the relation between these two objects following \cite{deng2022nilpotent}.
	
	\begin{defi}[Prolongation Bundles]
		Let $\overline{X}$ be a complex manifold and let $D = \sum_{i=1}^{\nu} D_{\nu}$ be an SNC divisor on $\overline{X}$. Let $(E, h)$ be a hermitian vector bundle on $X = \overline{X} \setminus D$. For each $\alpha= (\alpha_{1},\ldots, \alpha_{\nu}) \in \RR^{\nu}$, we can prolong $E$ to an $\cO_{\overline{X}}$-module $\cP_{\alpha} E$ as follows. Let $(U; z_{1},\ldots, z_{n})$ be an admissible coordinate and suppose that $D|_{U}$ is defined by $z_{1}\cdots z_{l} = 0$. For simplicity, assume that $D_{i} = (z_{i} = 0)$. Then,
		$$ \cP_{\alpha}E(U) = \left\{ \sigma \in \Gamma(E, X \cap U) : |\sigma|_{h} \lesssim \prod_{j=1}^{l} |z_{j}|^{-\alpha_{j} - \eps} \text{ on } U \text{ for all } \eps > 0  \right\}. $$
		For two non-negative functions $f$ and $g$, we write $f \lesssim g$ on $U$ if for each compact subset $K \subset U$, there exists some $C_{K} > 0$ such that $f \leq C_{K} g$ on $K$.
	\end{defi}
	
	We discuss several properties of prolongation bundles when $E$ comes from a complex polarized variation of Hodge structures. The following theorem is proved in \cite{mochizuki2002asymptotic} in the language of acceptable bundles.
	
	\begin{theo}[Mochizuki] \label{theo-Prolongation}
		Let $\overline{X}$ be a complex manifold and let $D = \sum_{i=1}^{\nu} D_{i}$ be an SNC divisor on $\overline{X}$. Let $X = \overline{X} \setminus D$ and $j \colon X \to \overline{X}$ the open inclusion. Let $E$ be a complex polarized variation of Hodge structures on $X$. For every $p, q$ and every $\alpha \in \RR^{\nu}$, the prolongation bundles $\cP_{\alpha} E^{p,q}$ satisfy the following properties:
		\begin{enumerate}
			\item $\cP_{\alpha}E^{p,q}$ is a locally free sheaf.
			\item $\cP_{\alpha} E \subset \cP_{\beta} E$ if $\alpha_{i} \leq \beta_{i}$ for all $i$.
			\item For $1 \leq i \leq \nu$, $\cP_{\alpha + e_{i}} E^{p,q} = \cP_{\alpha} E^{p,q} \otimes \cO_{X}(D_{i})$ where $e_{i}$ is the $i$-th unit vector.
			\item $\cP_{\alpha + \vec{\eps}} E^{p,q} = \cP_{\alpha}E^{p,q}$ for $\vec{\eps} = (\eps, \ldots, \eps)$ for all $0 < \eps \ll 1$.
			\item The set $\{ \alpha : \cP_{\alpha} E / \cP_{< \alpha} E  \neq 0\}$ is discrete in $\RR^{\nu}$. \hfill{$\square$}
		\end{enumerate}
	\end{theo}
	
	We finally introduce the construction of Deligne's extension and describe the relation to the prolongation via Hodge norms. Let $\overline{X}$ be a complex manifold and $D$ be an SNC divisor. Suppose we have a complex variation of Hodge structures $E$ on $X = \overline{X} \setminus D$. Choose an admissible coordinate $(\Omega; z_{1},\ldots, z_{n})$ and suppose that $D \cap \Omega$ is given by the equation $z_{1}\cdots z_{l} = 0$. The fundamental group $\pi_{1}(\Omega\sta)$ is isomorphic to $\ZZ^{l}$, with generators $\gamma_{1},\ldots, \gamma_{l}$, where $\gamma_{i}$ corresponds to the loop going once around the puncture counterclockwise in the $i$-th coordinate. The universal cover of $\Omega\sta$ is isomorphic to $\HH^{l} \times \Delta^{n-l}$ where $\HH = \{ z \in \CC: \mathrm{Im} z > 0\}$ is the upper-half plane. The quotient map is given by
	$$ \pi \colon \HH^{l} \times \Delta^{n-l} \to \Omega\sta , \qquad (w_{1},\ldots,w_{l}, z_{l+1}, \ldots,  z_{n}) \mapsto (e^{2\pi i w_{1}}, \ldots , e^{2\pi i w_{l}}, z_{l+1} ,\ldots, z_{n}).$$
	Let $V$ be the space of flat sections of $\pi\sta E$ on $\HH^{l} \times \Delta^{n-l}$. Then we have the monodromy representation $\Phi\colon \pi_{1}(\Omega\sta) \to \GL(V)$ given as follows.
	$$ \Phi(\gamma_{i}) (v) (w_{1}, \ldots, w_{l}, z_{l+1}, \ldots, z_{n}) = v(w_{1}, \ldots, w_{i} - 1, \ldots, w_{l}, z_{l+1}, \ldots, z_{n}).$$
	Denote by $T_{i}$ as the image of $\gamma_{i}$. The monodromy theorem says that the eigenvalues of $T_{i}$ have absolute value 1. Therefore, we can consider the Jordan decomposition
	$$ T_{j} = e^{2\pi i (S_{j} + N_{j})},$$
	where $S_{j}$ is a semisimple operator whose eigenvalues lie in the interval $[-\alpha_{j} , -\alpha_{j} + 1)$ and $N_{j}$ is a nilpotent operator. One can see that all $S_{j}$ and $N_{j}$ mutually commute since $\pi_{1}(\Omega\sta)$ is abelian. For a flat section $v \in V$, consider the expression
	$$ \tilde{v} =  \exp \left( \sum_{j=1}^{l}(S_{j} + N_{j})w_{j} \right) v. $$
	Then we see that $\tilde{v}$ is stable under the transformation $w_{j} \mapsto w_{j} + 1$ for all $1 \leq j \leq l$ hence, $\tilde{v}$ descends to a section on $\Omega\sta$. We define $E_{\alpha}$ on $\Omega$ to be the locally free sheaf generated by these sections $\tilde{v}$. For each subbundle $F^{p} E \subset E$, we define a filtration on $E_{\alpha}$ by
	$$ F^{p}E_{\alpha} \colon= j \lsta F^{p}E \cap E_{\alpha}.$$
	The nilpotent orbit theorem relates Deligne's extension and the prolongation bundles via norm growth. Here, we use the version of \cite{deng2022nilpotent} which is for complex polarized variations of Hodge structures.
	
	\begin{theo}[Nilpotent Orbit Theorem \cite{deng2022nilpotent}]
		Let $\overline{X}$ be a complex manifold and let $D = \sum_{i=1}^{\nu} D_{i}$ be a simple normal crossing divisor on $\overline{X}$. Let $E$ be a complex polarized variation of Hodge structures on $X = \overline{X} \setminus D$. Then for every $\alpha \in \RR^{\nu}$, $F^{p}E_{\alpha}$ and $E_{\alpha}^{p,q} = F^{p}E_{\alpha}/ F^{p+1} E_{\alpha}$ are locally free sheaves. Moreover, $E_{\alpha}^{p,q}$ can be naturally identified to $\cP_{\alpha}E^{p,q}$, which is the prolongation via growth of the Hodge norm. \hfill{$\square$}
	\end{theo}
	Note that the connection $\nabla \colon E \to \Omega_{X}^{1} \otimes E$ extends to a logarithmic connection
	$$ \nabla \colon E_{\alpha} \to \Omega_{\overline{X}}^{1} (\log D) \otimes E_{\alpha}$$
	and satisfies the Griffiths transversality condition, i.e., on each graded piece $E_{\alpha}^{p,q}$, the connection $\nabla$ induces an $\cO_{X}$-linear map
	$$ \theta \colon E_{\alpha}^{p,q} \to \Omega_{\overline{X}}^{1} (\log D) \otimes E_{\alpha}^{p-1,q+1}.$$

	\subsection{Some Standard Kähler Identities and Functional Analysis} This section is more or less standard, but we include it for more algebraically oriented readers. Let $(X, \omega)$ be a Kähler manifold. We denote by $L$ the Lefschetz operator $\omega \wedge \bullet$ and its adjoint by $\Lambda$ (or $\Lambda_{\omega}$ if we want to emphasize the Kähler metric). We will use the following Kähler identities which can be found in \cite{demailly2012analytic}*{Chapter 4.C} and many other standard textbooks in complex geometry and Hodge theory.
	
	\begin{prop} \label{prop-Kahler-comm}
		Let $E$ be a holomorphic vector bundle with a smooth hermitian metric on a Kähler manifold $(X, \omega)$. If $D = D' +D''$ is the Chern connection and $\delta'$, $\delta''$ are the corresponding adjoints of $D'$ and $D''$, then we have the following identities:
		\begin{align*}
			& [\delta'', L] = i D' & [\delta', L] & = - i D'' \\ &[\Lambda, D''] = -i\delta' & [\Lambda, D'] & = i \delta''.
		\end{align*}
	\end{prop}
	
	Also, we have the Bochner-Kodaira-Nakano identity for the Laplacian operator:
	
	\begin{prop} \label{prop-BKN}
		If $\Theta_{E}$ is the curvature of $E$, and
		$$ \lap'' = D'' \delta'' + \delta'' D'' \qquad \textnormal{and} \qquad \lap' = D' \delta' + \delta' D', $$
		then
		$$ \lap'' = \lap' + [i \Theta_{E}, \Lambda].$$
	\end{prop}
	
	We introduce a useful computation for the commutator operator $[i \Theta_{E}, \Lambda]$ in the case when $E$ is a line bundle. The following lemma can be obtained by a straightforward calculation.
	
	\begin{lemm} [\cite{demailly1997complex}*{Chapter VII.4}]
 \label{lemm-Girbau}
		Let $\omega$ be a Kähler metric on $X$ and let $\cL$ be a hermitian line bundle on $X$. If $\gamma_{1} \leq \ldots \leq \gamma_{n}$ are the eigenvalues of $i \Theta_{\cL,x}$ with respect to $\omega_{x}$, then we have
		$$ \langle [i \Theta_{\cL}, \Lambda ] u , u \rangle \geq (\gamma_{1} + \ldots + \gamma_{s} - \gamma_{r+1} - \ldots - \gamma_{n}) |u|^{2}$$
		for all $u \in \bigwedge^{r,s} T\sta_{x} X$.
	\end{lemm}
	
	{\it Proof.} Simultaneously diagonalize $\omega_{x}$ and $i \Theta_{\cL,x}$ so that we can express the two as
	$$
		 \omega_{x} = i \sum_{\mu} \zeta_{\mu} \wedge \bar{\zeta}_{\mu} \qquad \textnormal{and} \qquad  i \Theta_{\cL,x} = i \sum_{\mu} \gamma_{\mu} \zeta_{\mu} \wedge \bar{\zeta}_{\mu}.
	$$
	Then a straightforward computation tells us that for $u = \sum_{J, K} u_{J,K} \zeta_{J} \wedge \bar{\zeta}_{K}$, we have
	\begin{align*}
		\langle[i \Theta_{\cL}, \Lambda] u, u \rangle & = \sum_{J, K} \left( \sum_{j \in J} \gamma_{j} + \sum_{k \in K} \gamma_{k} - \sum_{l=1}^{n} \gamma_{l} \right) |u_{J, K}|^{2} \\
		& \geq (\gamma_{1} + \ldots + \gamma_{s} - \gamma_{r+1} - \ldots - \gamma_{n}) |u|^{2}. \qquad \square 
	\end{align*}
	
	The following is a standard lemma for $L^{2}$-existence statements. For example, see \cite{demailly1997complex}*{Chapter VIII. \S1}.
	
	\begin{lemm} \label{lemm-Hilbert-technique}
		Let $H_{1}, H_{2}$ and $H_{3}$ be Hilbert spaces and let $T \colon H_{1} \to H_{2}$ and $S \colon H_{2} \to H_{3}$ be closed and densely defined operators such that $ST = 0$. Let $T\sta$ and $S\sta$ be the adjoints of $T$ and $S$, respectively. Suppose that there exists $\eps > 0$ such that
		$$ \norm{T\sta u}^{2} + \norm{S u}^{2} \geq \eps^{2} \norm{u}^{2} \qquad \textnormal{for all } u \in \Dom(T\sta) \cap \Dom(S).$$
		Then for every $u \in H_{2}$ such that $Su =0$, there exists $v \in H_{1}$ such that $Tv = u$ and $\norm{v} \leq \eps^{-1} \norm{u}$.
	\end{lemm}

	\section{Proof of Theorem \ref{theo-Main}}
	
	\subsection{Dolbeault Resolution for the de Rham complex} \label{sec-deRham-Dolbeault}
	First, we ignore the issue of the SNC divisor and describe a nice resolution for $\gr^{p} \DR_{X}(E)$. Let $E$ be a complex polarized variation of Hodge structures on a Kähler manifold $X$. The graded de Rham complex $\gr^{p} \DR_{X}(E)$ of $E$ on $X$ is defined as follows
	$$  \Big[ E^{p, q} \xrightarrow{\theta} \Omega_{X}^{1} \otimes E^{p-1, q+1} \xrightarrow{\theta} \ldots \xrightarrow{\theta} \Omega_{X}^{n} \otimes E^{p-n, q+n} \Big] [n].$$
	
	By the relation $\db \theta + \theta \db = 0$, there is an anti-commuting morphism between each Dolbeault resolution. Therefore, the Dolbeault resolution for each sheaf $\Omega_{X}^{r} \otimes E^{p-r, q+r}$ combines to a double complex:
	
	$$ \begin{tikzcd}
		\cA_{X}^{0,n} \otimes E^{p,q} \ar[r, "\theta"]  & \cA_{X}^{1,n} \otimes E^{p-1, q+1} \ar[r, "\theta"] & \cdots \ar[r, "\theta"]  & \cA_{X}^{n,n} \otimes E^{p-n, q+n}\\
		\vdots \ar[r, "\theta"] \ar[u, "\db"] & \vdots \ar[r, "\theta"]\ar[u, "\db"] & \vdots \ar[r, "\theta"] \ar[u, "\db"] & \vdots \ar[u, "\db"] \\
		\cA_{X}^{0,1} \otimes E^{p,q} \ar[r, "\theta"] \ar[u, "\db"] & \cA_{X}^{1,1} \otimes E^{p-1, q+1} \ar[r, "\theta"] \ar[u, "\db"] & \cdots \ar[r, "\theta"] \ar[u, "\db"] & \cA_{X}^{n,1} \otimes E^{p-n, q+n} \ar[u, "\db"] \\
		\cA_{X}^{0,0} \otimes E^{p,q} \ar[r, "\theta"] \ar[u, "\db"] & \cA_{X}^{1,0} \otimes E^{p-1, q+1} \ar[r, "\theta"] \ar[u, "\db"] & \cdots \ar[r, "\theta"] \ar[u, "\db"] & \cA_{X}^{n,0} \otimes E^{p-n, q+n}. \ar[u, "\db"] \\
	\end{tikzcd}$$
	In particular, the total complex of this double complex is quasi-isomorphic to $\gr^{p} \DR_{X}(E)$. Consider
	$$ \bb{E}^{l} = \bigoplus_{r +s = l+n} \cA_{X}^{r, s} \otimes E^{p-r, q+r},$$
	which is the $l$-th entry of the total complex of this double complex and denote the boundary map by $\eth$. Then we obtain a complex
	$$ \bb{E}^{\bullet} : \bb{E}^{-n} \xrightarrow{\eth} \ldots \xrightarrow{\eth} \bb{E}^{0} \xrightarrow{\eth} \ldots \xrightarrow{\eth} \bb{E}^{n}$$
	with $\gr^{p}\DR_{X}(E) \simeq_{qis} (\bb{E}^{\bullet} , \eth)$. Each entry $\bb{E}^{l}$ has a natural hermitian structure by letting the direct summands be mutually orthogonal. After twisting by a line bundle $\cL$, we still have a quasi-isomorphism $\gr^{p} \DR_{X}(E) \otimes \cL \simeq_{qis} (\bb{E}^{\bullet} \otimes \cL, \eth)$.
	
	\subsection{More Identities} \label{sec-Identity}
	We introduce more identities which exploit the curvature formula for Hodge bundles. Throughout this section, all identities will be considered as formal identities, which means that the identity holds for all smooth and compactly supported sections.
	
	We fix a Kähler metric $\omega$ on $X$. Since $\cA_{X}^{r,s}$ has a hermitian structure, the operators $\nabla^{1,0}, \db, \theta,$ and $\varphi$ have adjoints which we denote by $(\nabla^{1,0})\sta $, $\db\sta$, $\theta\sta$ and $\varphi\sta$, respectively. These operators are of the following type
	\begin{align*}
		(\nabla^{1,0})\sta &\colon \cA_{X}^{r, s}(E^{p,q}) \to \cA_{X}^{r-1, s} (E^{p,q}), \\
		\db\sta & \colon \cA_{X}^{r,s} (E^{p,q}) \to \cA_{X}^{r, s-1} (E^{p,q}), \\
		\theta\sta &\colon \cA_{X}^{r, s}(E^{p,q}) \to \cA_{X}^{r-1, s} (E^{p+1, q-1}), \\
		\varphi\sta & \colon \cA_{X}^{r, s}(E^{p,q}) \to \cA_{X}^{r, s-1} (E^{p-1, q+1}).
	\end{align*}
	The first commutator relation is as follows.
	
	\begin{lemm} \label{lemm-Higgs-comm}
		With the above notation, we have
		$$
			 [\Lambda, \varphi] = i \theta\sta \qquad \textnormal{and} \qquad [\Lambda, \theta]= -i\varphi\sta.
		$$
	\end{lemm}
	{\it Proof.} All the operators are $\cC^{\infty}$-linear, hence we fix a point $x \in X$ and calculate everything pointwise. Therefore, we diagonalize the Kähler form $\omega$ and assume that $\omega = i \sum_{j=1}^{n} dz_{j} \wedge d\zbar_{j}$ at $x$ where $(z_{j})_{j=1}^{n}$ is a coordinate system centered at $x$. For $J = (j_{1},\ldots, j_{p})$ and $1 \leq l \leq n$, we use the following notation
	$$ dz_{l\lrcorner J} = i_{\de_{l}}(dz_{j_{1}}\wedge \cdots \wedge dz_{j_{p}}) = \begin{cases}
		0 & \textnormal{if } l \notin \{ j_{1},\cdots, j_{p}\} \\
		(-1)^{k-1} dz_{j_{1}} \wedge \cdots \wedge \widehat{dz_{j_{k}}} \wedge \cdots \wedge dz_{j_{p}} & \textnormal{if } l = j_{k}
	\end{cases}$$
	and
	$$ dz_{jJ} = dz_{j} \wedge dz_{J}. $$
	We define $d\zbar_{l \lrcorner K}$ and $d\zbar_{kK}$ analogously. Then for a $(p, q)$-form $u = \sum u_{J, K} dz_{J} \wedge d\zbar_{K}$, we have
	$$ \Lambda u = i (-1)^{|J|} \sum_{J, K, l} u_{J, K} dz_{l \lrcorner J} \wedge d\zbar_{l \lrcorner K}. $$
	After writing $\theta = \left(\sum_{j}  \theta_{j} dz_{j} \right) \wedge \bullet$, Lemma \ref{lemm-curvature} implies that the operators $\theta, \theta\sta, \varphi, \varphi\sta$ satisfy the following formulas:
	\begin{align*}
		\theta (u) & = \sum_{J, K} \sum_{j \in J^{c}} \theta_{j} (u_{J, K}) dz_{j J} \wedge d\zbar_{K} ,\\
		\theta\sta (u) & = \sum_{J, K} \sum_{j \in J} \theta_{j}\sta (u_{J, K}) dz_{j \lrcorner J} \wedge d\zbar_{K}, \\
		\varphi(u) & = (-1)^{|J|} \sum_{J, K} \sum_{k \in K^{c}} \theta_{k}\sta (u_{J, K}) dz_{J} \wedge d\zbar_{kK}, \\
		\varphi\sta(u) & = (-1)^{|J|} \sum_{J, K} \sum_{k \in K} \theta_{k}(u_{J, K}) dz_{J} \wedge d\zbar_{k \lrcorner K}.
	\end{align*}
	The rest of the proof is just brute force computation. If $u = \sum_{J, K} u_{J, K} dz_{J} \wedge d\zbar_{K}$, then
	\begin{equation} \label{eqaution-1}
	\begin{split}
				\Lambda \varphi (u) & = \Lambda \left((-1)^{|J|} \sum_{J, K} \sum_{k \in K^{c}} \theta_{k}\sta (u_{J, K}) dz_{J} \wedge d\zbar_{kK}  \right)\\
		& = i \Big[ \sum_{J, K} \sum_{k \in J \setminus K} \theta_{k}\sta (u_{J, K}) dz_{k \lrcorner J} \wedge d\zbar_{K} \\
		& \qquad + \sum_{J, K} \sum_{k \in K^{c}} \sum_{j \in J \cap K} \theta_{k}\sta (u_{J, K}) dz_{j \lrcorner J} \wedge d\zbar_{j \lrcorner kK} \Big].
	\end{split}
	\end{equation}
	
	Also, we have
	\begin{equation} \label{equation-2}
	\begin{split}
		\varphi \Lambda (u)& = i (-1)^{|J|} \varphi \left(\sum_{J, K} \sum_{j\in J\cap K} u_{J, K} dz_{j \lrcorner J} \wedge d\zbar_{j \lrcorner K} \right) \\
		& = -i \sum_{J, K} \sum_{j \in J \cap K} \Big[ \theta_{j}\sta (u_{J,K}) dz_{j \lrcorner J} \wedge d\zbar_{K} \\
		& \qquad + \sum_{k \in K^{c}} \theta_{k}\sta (u_{J, K}) dz_{j \lrcorner J} \wedge d\zbar_{k(j \lrcorner K)} \Big].
	\end{split}
	\end{equation}
	Using $d\zbar_{j \lrcorner kK} = - d\zbar_{k(j \lrcorner K)}$, we see that the second terms in \eqref{equation-2} and \eqref{eqaution-1} cancel each other out after subtracting and we have
	$$ (\Lambda \varphi - \varphi \Lambda)(u) = i \sum_{J, K} \sum_{k \in J} \theta_{k}\sta (u_{J, K}) dz_{k \lrcorner J} \wedge d\zbar_{K} = i \theta\sta (u).$$
	This concludes the first part of the proof. We similarly obtain the second equality, as follows:
	
	\begin{equation} \label{equation-3}
	\begin{split}
				\theta \Lambda u & = i (-1)^{|J|}\theta \left( \sum_{J, K} \sum_{j \in J \cap K} u_{J,K} dz_{j \lrcorner J} \wedge d\zbar_{j \lrcorner K} \right)\\
		& = i (-1)^{|J|} \Big[ \sum_{J, K} \sum_{j \in J \cap K} \theta_{j} (u_{J, K}) dz_{J} \wedge d\zbar_{j \lrcorner K} \\
		& \qquad + \sum_{J, K} \sum_{j \in J \cap K} \sum_{k \in J^{c}} \theta_{k} (u_{J, K}) dz_{k(j \lrcorner J)} \wedge d\zbar_{j \lrcorner K} \Big]
	\end{split}
	\end{equation}
	and
	\begin{equation} \label{equation-4}
	\begin{split}
		\Lambda \theta u & = \Lambda \left( \sum_{J, K} \sum_{k \in J^{c}} \theta_{k}(u_{J, K}) dz_{kJ} \wedge d\zbar_{K} \right) \\
		& = i (-1)^{|J| + 1} \Big[ \sum_{J, K} \sum_{k \in J^{c} \cap K} \theta_{k}(u_{J, K}) dz_{J} \wedge d\zbar_{k \lrcorner K} \\
		& \qquad + \sum_{J, K} \sum_{k \in J^{c}} \sum_{j \in J \cap K} \theta_{k}(u_{J, K}) dz_{j \lrcorner kJ} \wedge d\zbar_{j \lrcorner K} \Big].
	\end{split}
	\end{equation}

	Again, using $dz_{k (j \lrcorner J)} = - dz_{j \lrcorner kJ}$, we see that the second terms in \eqref{equation-3} and \eqref{equation-4} cancel each other out after subtracting and we get
	$$ (\Lambda \theta -\theta \Lambda ) u = i(-1)^{|J|+1} \sum_{J, K} \sum_{j \in K} \theta_{j} (u_{J, K}) dz_{J} \wedge d\zbar_{j \lrcorner K} =  -i \varphi\sta (u). \qquad \square $$

	\begin{lemm} \label{lemm-VHS-BKN}
		Let $E$ be a variation of Hodge structures on a Kähler manifold $(X, \omega)$. If $\Theta_{E^{p,q}}$ is the curvature operator of $E^{p,q}$ with respect to the Hodge metric, then we have the following identity:
		$$[ i\Theta_{E^{p,q}} , \Lambda] = - \theta\sta \theta - \theta \theta\sta + \varphi\varphi\sta + \varphi\sta \varphi.$$
	\end{lemm}
	{\it Proof.}  We compute	\begin{align*}
		[ i \Theta_{E^{p,q}} , \Lambda ] & = [ -i(\theta \varphi + \varphi \theta), \Lambda ] \qquad \because \text{Lemma } \ref{lemm-curvature} \\
		& = i (-\theta \varphi \Lambda - \varphi \theta \Lambda + \Lambda \theta \varphi + \Lambda \varphi \theta) \\
		& = i \Big( (- \theta \varphi \Lambda + \theta \Lambda \varphi) + (-\varphi \theta \Lambda + \varphi \Lambda \theta) + (\Lambda \theta \varphi - \theta \Lambda \varphi) + (\Lambda \varphi \theta - \varphi \Lambda \theta) \Big) \\
		& = i ( \theta [\Lambda, \varphi] + \varphi [\Lambda, \theta] + [\Lambda, \theta] \varphi + [\Lambda, \varphi] \theta ) \qquad \because \text{Lemma } \ref{lemm-Higgs-comm} \\
		& = - \theta \theta\sta + \varphi\varphi\sta + \varphi\varphi\sta - \theta\sta \theta.  \qquad \square
	\end{align*}
	
	\begin{prop} \label{prop-deRham-Apriori}
		Let $E$ be a complex polarized variation of Hodge structures on $X$ and let $\bb{E} $ be the Dolbeault resolution for $\gr^{p}\DR_{X}(E)$ as in Section \ref{sec-deRham-Dolbeault}. Let $\cL$ be a line bundle on $X$ with a smooth hermitian metric and let $\Theta_{\cL}$ be the curvature of $\cL$. For $\bfu = (u\uind{r})\in \bb{E}^{l} \otimes \cL$, where each $u\uind{r} \in \cA_{X}^{r, n+l-r}(E^{p-r, q+r}) \otimes \cL$, we have the following formal identity:
		$$  \norm{\eth \bfu}^{2} + \norm{\eth\sta \bfu}^{2} = \sum_{r} \langle [i \Theta_{\cL} , \Lambda] u\uind{r} ,u\uind{r} \rangle + \sum_{r} \norm{\nabla^{1,0} u\uind{r} + \varphi u\uind{r+1}}^{2} + \sum_{r} \norm{\varphi\sta u\uind{r} + (\nabla^{1,0})\sta u\uind{r+1}}^{2}. $$
	\end{prop}
	{\it Proof.} The expression $\langle \lap_{\eth} \bfu , \bfu \rangle$ has the diagonal terms and the off-diagonal terms as follows:
	\begin{align*}
		\norm{\eth \bfu}^{2} + \norm{\eth\sta \bfu}^{2} & = \sum_{r} \langle \lap_{\db} u\uind{r} + (\theta\theta\sta + \theta\sta \theta) u\uind{r} , u \uind{r} \rangle \\
		& \quad + \sum_{r} \langle (\theta \db\sta + \db\sta \theta) u\uind{r} , u\uind{r+1} \rangle + \langle  (\db \theta\sta + \theta\sta \db) u\uind{r+1}, u\uind{r} \rangle.
	\end{align*}
	First, we examine the diagonal term. The classical Bochner--Kodaira--Nakano identity gives
	$$ \lap_{\db} = \lap_{\nabla^{1, 0}} + [i \Theta_{E^{p, q} \otimes \cL} , \Lambda],$$
	where $\nabla^{1, 0} + \db$ is the Chern connection for the bundle $E^{p, q} \otimes \cL$. The curvature of $E^{p,q} \otimes \cL$ is given by
	$$ \Theta_{E^{p, q} \otimes \cL} = \id_{E^{p, q}} \otimes \Theta_{\cL} + \Theta_{E^{p, q}} \otimes \id_{\cL}.$$
	Summing all these term and using Lemma \ref{lemm-VHS-BKN}, we have
	\begin{align*}
		& \langle (\lap_{\db} + \theta\theta\sta + \theta\sta \theta) u\uind{r} , u\uind{r} \rangle\\
		&  = \langle [i\Theta_{\cL}, \Lambda] u\uind{r}, u\uind{r} \rangle + \norm{\nabla^{1,0} u\uind{r}}^{2} + \norm{ (\nabla^{1, 0})\sta u\uind{r} }^{2} + \norm{\varphi u\uind{r}}^{2} + \norm{\varphi\sta u\uind{r}}^{2} .
	\end{align*}
	We next examine the off-diagonal term. Note that 
	\begin{equation*}
		\langle (\theta \db\sta + \db\sta \theta)u\uind{r} , u\uind{r+1} \rangle = \overline{ \langle  (\db \theta\sta + \theta\sta \db) u\uind{r+1}, u\uind{r} \rangle }
	\end{equation*}
 	and we have
	\begin{align*}
		&\db \theta \sta + \theta \sta \db \\
		& = -i (\db \Lambda \varphi - \db \varphi \Lambda + \Lambda \varphi\db - \varphi \Lambda \db) \qquad \because\text{Lemma \ref{lemm-Higgs-comm}} \\
		& = (-i) \left( \Lambda \db \varphi + i (\nabla^{1,0})\sta \varphi - \db \varphi \Lambda + \Lambda \varphi \db - \varphi \db \Lambda + i \varphi (\nabla^{1,0})\sta \right) \qquad \because [\Lambda, \db] = -i (\nabla^{1,0})\sta \\
		& = (\nabla^{1,0})\sta \varphi + \varphi (\nabla^{1,0})\sta \qquad \because \db \varphi + \varphi \db = 0.
	\end{align*}
	Therefore, 
	$$ \langle u\uind{r}, (\db \theta\sta + \theta\sta \db) u\uind{r+1} \rangle = \langle \nabla^{1,0} u\uind{r} , \varphi u\uind{r+1} \rangle + \langle \varphi\sta u\uind{r} , (\nabla^{1,0})\sta u\uind{r+1} \rangle. $$
	Adding all up, we get the desired identity. \hfill{$\square$}

	\subsection{$L^{2}$-existence results}
	The following $L^{2}$-existence result will serve as a key tool for the proof of cohomology vanishing. The first proposition is a global version of the $L^{2}$-existence result. One can also find a similar result in \cite{deng2022vanishing}*{Corollary 2.7} in the language of harmonic bundles, but we prefer to include the proof for completeness.
	
	\begin{prop} \label{prop-L2-global}
		Let $(X, \omega)$ be a Kähler manifold (not necessarily compact) and let $\cL$ be a line bundle with a smooth hermitian metric $h_{\cL}$. Let $E$ be a complex polarized variation of Hodge structures on $X$. Furthermore, we assume that
		\begin{enumerate}
			\item The geodesic distance $\delta_{\omega}$ is complete on $X$, i.e., $(X, \delta_{\omega})$ is complete as a metric space.
			\item The norm of the Higgs field $|\theta|_{h, \omega}^{2}$ is globally bounded.
			\item There exists $\eps > 0$ such that $\langle [i\Theta_{\cL}, \Lambda_{\omega}] u, u \rangle \geq  \eps \norm{u}_{\omega}^{2}$ for all smooth and compactly supported $(r, s)$-forms $u$ for $r + s = n + l$.
		\end{enumerate}
		Let $\bfu$ be a measurable section with values in $\bb{E}^{l}$ such that $\eth \bfu = 0$. Provided that the right hand side of the expression below is finite, there exists a measurable section $\mathbf{v}$ with values in $\bb{E}^{l-1}$ satisfying $\eth \mathbf{v} = \bfu$ and the following inequality
		$$ \int_{X} \norm{\mathbf{v}}_{h, \omega}^{2} dV_{\omega} \leq \frac{1}{\eps} \int_{X} \norm{\bfu}_{h, \omega}^{2} dV_{\omega}.$$
	\end{prop}
	{\it Proof.} The only non-trivial part is to approximate measurable sections $\bfu$ by smooth and compactly supported sections in a desirable way. Note that we have an a priori inequality $\norm{\eth \bfu}^{2} + \norm{\eth\sta \bfu}^{2} \geq \eps \norm{\bfu}^{2}$ for smooth and compactly supported sections. Since $|\theta|_{h, \omega}^{2}$ is globally bounded, 
	$$ \theta \colon L^{2} (X, \cA_{X}^{r, s} \otimes E^{p,q}) \to L^{2}(X, \cA_{X}^{r+1, s} \otimes E^{p,q})$$
	is a bounded operator. Therefore, for $L^{2}$-sections $\bfu = (u\uind{r})$, the section $\eth \bfu$ is $L^{2}$ if and only if each derivative $\db u\uind{r}$ is $L^{2}$. Similarly, $\eth\sta \bfu$ is $L^{2}$ if and only if for each $r$, the section $\db\sta u\uind{r}$ is $L^{2}$. Moreover, provided a sequence $\bfu_{\nu} \to \bfu$ in $L^{2}$, we have $\eth \bfu_{\nu} \to \eth \bfu$ in $L^{2}$ if and only if for each $r$, $\db u_{\nu}\uind{r} \to \db u\uind{r}$ in $L^{2}$. Similarly, $\eth\sta \bfu_{\nu} \to \eth\sta \bfu$ in $L^{2}$ if and only if $\db\sta u_{\nu}\uind{r} \to \db\sta u\uind{r}$ in $L^{2}$ for each $r$. As in the proof of \cite{demailly2012analytic}*{Theorem 5.1}, the completeness of $(X, \delta_{\omega})$ tells us that the space of compactly supported smooth forms is dense in $\Dom(\db) \cap \Dom(\db\sta)$ under the graph norm
	$$ u \mapsto (\norm{u}^{2} + \norm{\db u}^{2}  +\norm{\db\sta u}^{2})^{1/2} .$$
	Therefore, the a priori inequality 
	$$ \norm{\eth \bfu}^{2} + \norm{\eth\sta \bfu}^{2} \geq \eps \norm{\bfu}^{2}$$
	also holds for measurable sections such that $\bfu$, $\eth\bfu$, and $\eth\sta \bfu$ are all $L^{2}$. The remaining part of the proof follows from Lemma \ref{lemm-Hilbert-technique}. \hfill{$\square$}
	
	Also, we have a local version of the $L^{2}$-existence result which will be used in the construction of the $L^{2}$-Dolbeault resolution. The following result is motivated by \cite{sabbah2022degenerating}*{\S 159}, using the fact that even though the curvature of a Hodge bundle is not positive, one can bound the negative contribution of the curvature by the Higgs field estimates. Hence, after twisting with a suitable weight, one can solve the $\db$-equation.
	
	\begin{prop} \label{prop-L2-local}
		Equip $\Omega\sta = (\Delta\sta)^{l} \times \Delta^{n-l}$ with the Poincaré metric $\omega_{\Poin}$ defined in Section \ref{sec-Higgs-field-bound}. Let $E$ be a complex polarized variation of Hodge structures on $\Omega\sta$. Let $C > 0$ be a number such that $\norm{\theta}^{2}_{\omega_{\Poin}} < C$ on $\Omega\sta$. Define $\eta\colon\Omega\sta \to \RR$ by the following formula, for $a_{j} \in \RR$ and $b_{j} > C + 2$:
		$$ e^{-\eta} = \prod_{j = 1}^{l} |z_{j}|^{2a_{j}} (-\log |z_{j}|^{2})^{b_{j}} \prod_{j = l+1}^{n} e^{- b_{j} |z_{j}|^{2}}.$$
		If $u$ is an $(r, s)$-form with values in $E^{p,q}$, with measurable coefficients such that $\db u = 0$ and
		$$ \int_{\Omega\sta} \norm{u}_{h, \omega_{\Poin}}^{2} e^{-\eta} dV_{\Poin} < + \infty,$$
		then there exists an $(r, s-1)$-form with values in $E^{p,q}$, with measurable coefficients, such that $u = \db v$ and
		$$ \int_{\Omega\sta} \norm{v}_{h, \omega_{\Poin}}^{2} e^{-\eta} dV_{\Poin} \leq \int_{\Omega\sta} \norm{u}_{h, \omega_{\Poin}}^{2} e^{-\eta} dV_{\Poin}.$$
	\end{prop}
	{\it Proof.} First, we prove the result for $(n, s)$-forms. A straightforward computation gives
	$$ i \de \db \eta = i \sum_{j=1}^{l} \frac{b_{j}}{|z_{j}|^{2} ( -\log |z_{j}|^{2})^{2}} dz_{j} \wedge d\zbar_{j} + i \sum_{j = l+1}^{n} b_{j} dz_{j} \wedge d\zbar_{j}.$$
	We put $dz\lbind{n} = dz_{1} \wedge \cdots \wedge dz_{n}$. If $u = \sum_{K} u_{K} dz\lbind{n} \wedge d\zbar_{K}$, we have
	$$ [ i \de \db \eta , \Lambda_{\omega_{\Poin}} ] u = \sum_{K} \left( \sum_{j \in K} b_{j}\right) u_{K} dz\lbind{n} \wedge d\zbar_{K} .$$
	Using the Bochner--Kodaira--Nakano identity with a twist by $e^{-\eta}$, we have the a priori inequality
	\begin{align*}
		\norm{\db u}^{2} + \norm{\db\sta u}^{2} & \geq \langle  [i \Theta_{E} , \Lambda_{\omega_{\Poin}}] u, u \rangle + \langle [ i \de \db \eta , \Lambda_{\omega_{\Poin}}] u , u \rangle \\
		& \geq - \langle \theta \theta\sta u, u \rangle + \langle [i \de \db \eta, \Lambda_{\omega_{\Poin}}] u, u \rangle \qquad \because \text{Lemma } \ref{lemm-VHS-BKN} \\
		& \geq - C \norm{u}^{2}+ \langle [i \de \db \eta, \Lambda_{\omega_{\Poin}}] u, u \rangle.
	\end{align*}
	Here, the adjoint operator $\db\sta$ is computed with respect to the metric $h_{E} \cdot e^{-\eta}$ on the Hodge bundles. In particular, if $b_{j} \geq C+ 1$ for each $j$, then we have $\norm{\db u}^{2} + \norm{\db\sta u}^{2} \geq \norm{u}^{2}$. The assertion for $(n, s)$-forms now follows from Lemma \ref{lemm-Hilbert-technique}. 
	One technical part is that $\omega_{\Poin}$ is not a complete metric on the domain $\Omega\sta$. However, we can remedy this situation since $\Omega\sta$ is weakly pseudoconvex. This can be done by taking a plurisubharmonic exhaustion function $\psi \geq 0$ of $\Omega\sta$ and considering the complete metric
	$$ \omega_{\Poin, \eps} = \omega_{\Poin} + i \eps \de \db \psi^{2}. $$
	Then we get the $L^{2}$-existence theorem for $\omega_{\Poin}$ by applying the corresponding result for each $\omega_{\Poin, \eps}$ and taking the weak limit. This is the strategy in \cite{demailly1982estimations}*{Théorème 4.1}.
	
	We next prove the result for $(r, s)$-forms. Let $u = \sum_{J, K} u_{J, K} dz_{J} \wedge d\zbar_{K}$. We put 
	$$ u\lind{J} = \sum_{K} u_{J, K} dz_{J} \wedge d\zbar_{K} \qquad \text{and} \qquad \tilde{u}\lind{J} = \sum_{K} u_{J, K} dz\lbind{n} \wedge d\zbar_{K} .$$
 	Define $\eta'$ by
	$$ e^{-\eta'} = e^{-\eta} \prod_{j \notin J, j \leq l} |z_{j}|^{-2} (-\log |z_{j}|^{2})^{-1} = \prod_{j = 1}^{l} |z_{j}|^{2a_{j}'} (-\log |z_{j}|^{2})^{b_{j}'} \prod_{j = l+1}^{n} e^{- b_{j}' |z_{j}|^{2}}. $$
	Note that $ \norm{u\lind{J}}_{h, \omega_{\Poin}}^{2} e^{-\eta} = \norm{\tilde{u}\lind{J}}_{h, \omega_{\Poin}}^{2} e^{-\eta'} $. We have $\int_{\Omega\sta} \norm{\tilde{u}\lind{J}}_{h, \omega_{\Poin}}^{2} e^{-\eta'} dV_{\Poin} < + \infty$. Since $b_{j}' > C+1$, we can apply the result for $(n, s)$-forms so that there exists an $(n, s-1)$-form $\tilde{v}\lind{J}$ such that $\db \tilde{v}\lind{J} = \tilde{u}\lind{J}$ and
	$$ \int_{\Omega\sta} \norm{\tilde{v}\lind{J}}_{h, \omega_{\Poin}}^{2} e^{-\eta'} dV_{\Poin} \leq \int_{\Omega\sta} \norm{\tilde{u}\lind{J}}_{h, \omega_{\Poin}}^{2} e^{-\eta'} dV_{\Poin}.$$
	Since wedging by holomorphic coordinate does not effect the operator $\db$ (up to sign), there exists an $(r, s-1)$-form $v\lind{J}$ such that $\db v\lind{J} = u\lind{J}$ and
	$$ \int_{\Omega\sta} \norm{v\lind{J}}_{h, \omega_{\Poin}}^{2} e^{-\eta} dV_{\Poin} \leq \int_{\Omega\sta} \norm{u\lind{J}}_{h, \omega_{\Poin}}^{2} e^{-\eta} dV_{\Poin}.$$
	At the end, $\sum_{J} v\lind{J}$ gives the desired solution. \hfill{$\square$}
	
	\subsection{Construction of the Kähler metric on the complement} \label{sec-Kahlermetriconopen}
	From now on, we work in the setting of the main theorem. Hence $D  = \sum_{i=1}^{\nu} D_{i}$ is an SNC divisor on a compact Kähler manifold $\overline{X}$, and $X = \overline{X} \setminus D$. We cover $D$ with finitely many admissible coordinates $\{ \Omega_{i}\}_{i \in I}$ and use Theorem \ref{theo-HiggsBound} to get a real number $C > 0$ such that $|\theta|_{h, \omega_{\Poin}}^{2} < C$ for each admissible coordinate $\Omega_{i}$. Also, $\cL$ is a line bundle on $\overline{X}$ such that $\cL + \sum_{i=1}^{\nu} \alpha_{i} D_{i}$ has a smooth hermitian metric with semi-positive curvature, and at each point $x \in \overline{X}$,  the curvature has at least $n -t$ positive eigenvalues. In addition, $B$ is a nef line bundle on $\overline{X}$. The goal is to construct a Kähler metric $\omega$ on the open locus $X$ so that the following four conditions are satisfied:
	\begin{enumerate} 
		\item $(X, \omega)$ is complete.
		\item The norm of the Higgs field $|\theta|_{h, \omega}^{2}$ is globally bounded.
		\item If we write the curvature of $\cL \otimes B$ as $i \Theta_{\cL \otimes B}$, then the operator $[ i \Theta_{\cL \otimes B}, \Lambda_{\omega} ]$ should be positive definite for $(r, s)$-forms.
		\item We can solve a local $\db$-equation in admissible coordinates with an appropriate twist.
	\end{enumerate}
	The first three conditions will allow us to use the global $L^{2}$-existence result. Also, the last condition is crucial since we want to guarantee the exactness of a Dolbeault type complex. The correct guess based on the second and the fourth condition is that the Kähler metric $\omega$ should somehow look like the Poincaré metric for each admissible coordinate (see Theorem \ref{theo-HiggsBound} and Proposition \ref{prop-L2-local}). However, we cannot get the positivity of the operator $[i \Theta_{\cL \otimes B} , \Lambda_{\omega}]$ by using the metric that is naïvely constructed. The idea is to twist the metric of $\cL$ by a certain amount so that all conditions are true. The construction is essentially identical to \cite{deng2022vanishing}*{\S3.5}, modulo the effect of the twist by $\sum \alpha_{i} D_{i}$.
	
	\begin{prop} \label{prop-constr-Kahler}
		 With the notation as above, there exists a Kähler metric $\omega$ on $X$ and smooth hermitian metrics $h_{\cL}'$ on $\cL|_{X}$ and $h_{B}$ on $B$ satisfying the following:
		\begin{enumerate}
			\item $(X, \omega)$ is complete.
			\item $\omega$ and $\omega_{\Poin}$ are mutually bounded for each admissible coordinate $\Omega_{i}$.
			\item $|\theta|_{h, \omega}^{2}$ is globally bounded.
			\item If $\Theta_{\cL \otimes B}'$ is the curvature associated to the metric $h_{\cL}' \otimes h_{B}$, then we have the inequality $ \langle [i \Theta_{\cL\otimes B}' , \Lambda_{\omega}] u , u \rangle \geq \frac{1}{10} |u|_{\omega}^{2}$ for $(r,s)$-forms on $X$ with $r + s > n + t$.
		\end{enumerate}
	\end{prop}
	
		{\it Proof.} First, we fix a Kähler metric $\omega_{0}$ on $\overline{X}$. Choose hermitian metrics on $\cL$ and $\cO_{\overline{X}}(D_{i})$ and denote the corresponding curvatures (multiplied by $i$) by $\omega_{\cL}$, and $\omega_{1},\cdots, \omega_{\nu}$. For each $x \in X$, we can simultaneously diagonalize $\omega_{0}$ and $\omega_{\cL} + \sum_{j=1}^{\nu} \alpha_{j} \omega_{j}$ such that
	$$ \omega_{0}(x) = i \sum_{\mu} \zeta_{\mu} \wedge \bar{\zeta}_{\mu} \qquad \textnormal{and} \qquad \omega_{\cL} + \sum_{j=1}^{\nu} \alpha_{j} \omega_{j} = i \sum_{\mu} \gamma_{\mu}(x) \zeta_{\mu} \wedge \bar{\zeta}_{\mu}. $$
	The condition on the curvature of $\cL + \sum_{j=1}^{\nu} \alpha_{j} D_{j}$ guarantees that
	$$ 0 \leq \gamma_{1}(x) \leq \cdots \leq \gamma_{n}(x)$$
	and $\gamma_{t}(x) > 0$ since there are at least $n-t$ positive eigenvalues. Also, the $\gamma_{i}$ are continuous functions. Let
	$$ m = \min_{x \in \overline{X}} \gamma_{t}(x) > 0.$$
	Choose $0 < \eps_{1} < \eps$ such that
	$$ t \cdot \frac{\eps_{1} - \eps}{\eps_{1}} + (s-t) \cdot \frac{m + (\eps_{1} - \eps)}{m + \eps_{1}} - (n- r) > \frac{1}{10}$$
	for all $r + s > n + t$. This can be done since the above expression is identical to
	$$ (r+ s - n - t) - \frac{\eps_{1}}{m + \eps_{1}} (s-t) - (\eps - \eps_{1}) \left(\frac{t}{\eps_{1}} + \frac{s-t}{m + \eps_{1}}\right). $$
	Let $2\eps_{2} = \eps - \eps_{1} > 0$. Fix $\delta > 0$ such that
	$$ \omega_{0}+ \sum_{j=1}^{\nu} \delta_{j} \omega_{j} > 0 \qquad \text{for all } \delta_{j} \in [-\delta, \delta].$$
	Pick $\alpha_{j} < a_{j} < \alpha_{j} + \frac{\delta \eps_{2}}{2}$ and $b_{j} = C + 3$. Let $\sigma_{j} \in H^{0}(\overline{X}, \cO_{\overline{X}}(D_{j}))$ that vanishes along $D_{j}$. Also, take $a_{j}$ to be sufficiently close to $\alpha_{j}$ so that $\cP_{\alpha}E^{p,q} = \cP_{a}E^{p,q}$. We can rescale $\sigma_{j}$ so that $|\sigma_{j}|_{h_{j}}^{2} \leq \exp (-2(C + 3)/\delta \eps_{2})$. Fix a smooth hermitian metric $h_{B}$ on $B$ so that $i \Theta_{B} + \eps_{2} \omega_{0} > 0$. This can be done since $B$ is a nef line bundle. Denote $\omega_{B} = i \Theta_{B}$. Let
	$$ \tau_{j} \colon X \to \RR, \qquad \tau_{j}(x) = - \log|\sigma_{j}|_{h_{j},x}^{2}$$
	and let
	$$ \xi = \sum_{j=1}^{\nu} a_{j} \tau_{j} - b_{j} \log \tau_{j}.$$
	Define the Kähler metric $\omega$ on $X$ as
	$$\omega = \eps \omega_{0} + \omega_{\cL} + \omega_{B} + i \de \db \xi.$$
	The computation of $\de \db \xi$ gives
	$$ i \de \db \xi = \sum_{j=1}^{\nu} \left(a_{j} - \frac{b_{j}}{\tau_{j}} \right) \omega_{j} + \sum_{j=1}^{\nu} i \frac{b_{j}}{\tau_{j}^{2}} \de \tau_{j} \wedge \db \tau_{j}, $$
	using $i \de \db \tau_{j} = \omega_{j}$. Therefore, we can decompose $\omega$ into five terms
	$$ \omega = \eps_{1} \omega_{0}+ \left(\omega_{\cL} + \sum_{j=1}^{\nu} \alpha_{j} \omega_{j} \right) + \eps_{2} \left(\omega_{0}+ \sum_{j} \delta_{j} \omega_{j}\right) + (\eps_{2} \omega_{0}+ \omega_{B}) + i \sum_{j=1}^{\nu} \frac{b_{j}}{\tau_{j}^{2}} \de \tau_{j} \wedge \db \tau_{j} \qquad \mathbf{(A)}.$$
	Here, we have
	$$ \delta_{j} = \eps_{2}^{-1} \left( a_{j} - \alpha_{j} - \frac{b_{j}}{\tau_{j}} \right) \in [-\delta, \delta] $$
	since $\eps_{2}^{-1} (a_{j} - \alpha_{j}) \in [0, \delta/2]$ and $\tau_{j} \geq 2(C + 3) / \delta \eps_{2}$. Hence, we see that all the five terms are semi-positive and $\eps_{1} \omega_{0} > 0$. Therefore, $\omega$ is a Kähler form on $X$.
	
	Now we prove that the metric $\omega$ and $\omega_{\Poin}$ are mutually bounded for each admissible coordinate $\Omega_{i}$. Clearly, this implies that $(X, \omega)$ is complete. Fix an admissible coordinate $(\Omega; z_{1},\cdots, z_{n})$ and assume that $D \cap \Omega$ is defined by the equation $z_{1}\cdots z_{l} = 0$. For convenience, suppose that $(z_{i}= 0) = D_{i} \cap \Omega$. Note that $\tau_{i} = - \log |z_{i}|^{2} + g_{i}$ for some smooth function $g_{i}$ on $\Omega$. Then 
	$$ \frac{b_{j}}{\tau_{j}^{2}} \de \tau_{j} \wedge \db \tau_{j} = \frac{b_{j}}{ (-\log |z_{j}|^{2} + g)^{2} } \left( \frac{dz_{j}}{z_{j}} - \de g_{j} \right) \wedge \left( \frac{d\zbar_{j}}{\zbar_{j}} - \db g_{j} \right). $$
	Since the first four terms on the right hand side of {\bf (A)} are smooth on the entire space $\overline{X}$, we see that $\omega$ and $\omega_{\Poin}$ are mutually bounded on $\Omega$. This also shows that $|\theta|_{h, \omega}^{2}$ is globally bounded.
	
	The last thing to check is the positivity of the commutator operator $[i \Theta_{\cL \otimes B} , \Lambda_{\omega}]$. Here, the key idea is to twist the metric of $\cL$ by an extra factor of $e^{-\xi}$. This gives us a smooth hermitian metric $h_{\cL}' = h_{\cL}e^{-\xi}$ on $\cL|_{X}$. We denote the corresponding curvature as $\Theta_{\cL \otimes B}'$. The formula for this curvature is
	$$ i \Theta_{\cL \otimes B}' = \omega_{\cL} + \omega_{B} + i \de \db \xi.$$
	Simultaneously diagonalize $\omega_{0}$ and $i \Theta_{\cL \otimes B}'$ at $x\in X$ and express
	$$
		\omega_{0} = i \sum_{\mu} \zeta_{\mu} \wedge \bar{\zeta}_{\mu} \qquad \textnormal{and} \qquad  i \Theta_{\cL \otimes B}' = i \sum_{\mu} \gamma_{\mu}'(x) \zeta_{\mu} \wedge \bar{\zeta}_{\mu}.
	$$
	Also, we assume $\gamma_{1}' \leq \cdots \leq \gamma_{n}'$. Using {\bf (A)} and the fact that $\omega = \eps \omega_{0} + i \Theta_{\cL \otimes B}'$, we see that
	$$ \eps + \gamma_{\mu}'(x) \geq \eps_{1} + \gamma_{\mu}(x).$$
	If we diagonalize $i \Theta_{\cL \otimes B}'$ with respect to $\omega$ and denote the eigenvalues as $\gamma_{\mu, \eps}'(x)$, then we have
	$$ \gamma_{\mu, \eps}' = \frac{\gamma_{\mu}'}{\gamma_{\mu}' + \eps}.$$
	Note that $\gamma_{\mu}'(x) \geq \eps_{1} - \eps + \gamma_{\mu}(x) \geq \eps_{1} - \eps$. Also, if $\mu \geq t$, then we have the improved bound $\gamma_{\mu}'(x) \geq \eps_{1} - \eps + m$. Hence, we have the following bounds for the quantities for $\gamma_{\mu, \eps}'$:
	\begin{align*}
		& 1 \geq \gamma_{\mu, \eps}' \geq \frac{\eps_{1} - \eps}{\eps_{1}} \qquad \text{ for } 1 \leq \mu \leq t - 1 ,\\
		& 1 \geq \gamma_{\mu, \eps}' \geq \frac{m + \eps_{1} - \eps}{m + \eps_{1}} \qquad \text{ for } t \leq \mu \leq n.
	\end{align*}
	Using Lemma \ref{lemm-Girbau} for $(r, s)$-forms $u$, we have
	\begin{align*}
		\langle [i \Theta_{\cL \otimes B}' , \Lambda_{\omega}] u, u \rangle_{x} & \geq (\gamma_{1,\eps}'(x) + \cdots + \gamma_{s, \eps}'(x) - \gamma_{r+1, \eps}'(x) - \cdots - \gamma_{n, \eps}'(x)) | u|_{\omega, x}^{2} \\
		& \geq \Big[ \left(\frac{\eps_{1} - \eps}{\eps_{1}}\right)\cdot t + \left(\frac{m + \eps_{1} - \eps}{m + \eps_{1}}\right) \cdot (s - t) - (n-r) \Big] |u|_{\omega, x}^{2} \\
		& \geq \frac{1}{10} |u|_{\omega, x}^{2}. \qquad \square
	\end{align*}
	
	\begin{rema}
	The construction of the Kähler metric $\omega$ becomes simpler when $\cL + \sum \alpha_{i}D_{i}$ is ample and $B = \cO_{\overline{X}}$. We pick smooth hermitian metrics $h_{\cL}$ and $h_{i}$ on $\cL$ and $\cO_{\overline{X}}(D_{i})$ such that $\omega_{\cL} + \sum \alpha_{i}\omega_{i}$ is a Kähler metric on $\overline{X}$. Then the identical construction of $\xi : X \to \RR$ (with appropriate coefficients $a_{i}$ and $b_{i}$) gives the complete Kähler metric $\omega = \omega_{\cL} + i \de \db \xi$ on $X$ that we wanted. Also, the curvature $i \Theta_{\cL}'$ associated to the metric $h_{\cL}e^{-\xi}$ is exactly $\omega$. Therefore, $[i \Theta_{\cL}', \Lambda_{\omega}]u = (r + s - n)u$ for $(r, s)$-forms $u$. Hence, the positivity of $[i \Theta_{\cL}', \Lambda_{\omega}]$ is obvious.
	\end{rema}
	
	\subsection{$L^{2}$-Dolbeault resolution} \label{sec-L2-Dolbeault}
	The construction of the Kähler metric on $X$ in the previous section allows us to construct a Dolbeault-type resolution of the logarithmic de Rham complex of a variation of Hodge structures defined on the complement of an SNC divisor, extending the resolution described in \S3.1 on $X$. The explicit construction is as follows. Consider the Kähler metric $\omega$ on $X$ and the smooth hermitian metric on $E \otimes \cL \otimes B|_{X}$ given by $h' = (h_{E} \otimes h_{\cL} \otimes h_{B} )\cdot e^{-\xi}$ where $\omega$, $h_{\cL}' = h_{\cL}e^{-\xi}$ and $h_{B}$ are the metrics that we constructed in Proposition \ref{prop-constr-Kahler}. We define $\cH\lind{r}^{2}(E^{p,q} \otimes \cL \otimes B)$ to be the sheaf on $\overline{X}$ whose section on $U \subset \overline{X}$ are the holomorphic sections $u$ of $\Omega_{X}^{r}\otimes E^{p,q} \otimes \cL \otimes B$ on $U \cap X$ such that $\norm{u}_{h', \omega}^{2}$ is locally integrable with respect to the volume form $dV_{\omega}$. Similarly, define $L\lind{r, s}^{2} (E^{p,q} \otimes \cL \otimes B)$ to be the sheaf on $\overline{X}$ whose sections $u$ on $U \subset \overline{X}$ are $(r, s)$-forms with values in $E^{p,q} \otimes \cL \otimes B$ with measurable coefficients on $U \cap X$ such that both $\norm{u}_{h', \omega}^{2} $ and $\norm{\db u}_{h', \omega}^{2}$ are locally integrable with respect to the volume form $dV_{\omega}$. Then we can consider the following complex:
	$$0 \to \cH\lind{r}^{2} (E^{p, q} \otimes \cL\otimes B) \to L_{(r, 0)}^{2} (E^{p,q} \otimes \cL\otimes B) \xrightarrow{\db} L_{(r, 1)}^{2} (E^{p,q} \otimes \cL\otimes B)\xrightarrow{\db} \cdots\xrightarrow{\db} L_{(r, n)}^{2}(E^{p,q}\otimes \cL \otimes B).$$
	A priori, it is not clear that this complex is exact.
	
	\begin{prop} \label{prop-L2-deRham}
		The complex
		$$ 0 \to \cH\lind{r}^{2} (E^{p,q} \otimes \cL \otimes B) \to L\lind{r, \bullet}^{2} (E^{p,q} \otimes \cL \otimes B)$$
		is a resolution of $\cH\lind{r}^{2} (E^{p,q} \otimes \cL \otimes B)$ by fine sheaves.
	\end{prop}
	{\it Proof.} The exactness at $\cH\lind{r}^{2} (E^{p,q} \otimes \cL \otimes B)$ is clear by the regularity of the $\db$-equation. Indeed, if a local section $u$ on $L_{(r,0)}^{2}(E^{p,q}\otimes \cL \otimes B)$ satisfies $\db u = 0$, then $u$ is holomorphic. Also, it is clear that the sequence is exact on the open locus $X$. For the exactness on the boundary, we need to solve a $\db$-equation on a domain of type $\Omega\sta = (\Delta\sta)^{l} \times \Delta^{n-l}$. By construction, we can cover $D$ with admissible coordinates $\Omega \subset \overline{X}$ such that $|\theta|_{h, \omega_{\Poin}}^{2} < C$. Let $D \cap \Omega$ be defined by $z_{1}\cdots z_{l} = 0$. For simplicity, assume $D_{i} \cap \Omega = (z_{i} = 0)$. Also, assume that $\cL \otimes B$ is locally trivial on $\Omega$ and is trivialized by a non-vanishing section $\sigma$. Let $u \otimes \sigma \in L_{(r, s)}^{2} (E^{p,q} \otimes \cL \otimes B)(\Omega)$ such that $\db u = 0$. Let $\eta \colon \Omega\sta \to \RR$ be such that
	$$ e^{-\eta} = \prod_{i=1}^{l} |z_{i}|^{-2a_{i}} (- \log |z_{i}|^{2})^{b_{i}}.$$
	Possibly after shrinking the domain, we can assume that
	$$ \int_{\Omega\sta}\norm{u}_{h_{E}, \omega_{\Poin}}^{2} e^{-\eta} dV_{\omega_{\Poin}} < + \infty,$$
	since $\omega_{\Poin}$ and $\omega$ are mutually bounded on $\Omega\sta$ and $e^{-\eta}$ and $e^{-\xi}$ are mutually bounded on $\Omega\sta$. Hence, there exists $v$ such that $\db v = u$ and
	$$ \int_{\Omega\sta} \norm{v}_{h_{E}, \omega_{\Poin}}^{2} e^{-\eta} dV_{\omega_{\Poin}} < + \infty$$
	by Proposition \ref{prop-L2-local}. This implies that $v \otimes \sigma \in L_{(r, s-1)}^{2}(E^{p,q} \otimes \cL\otimes B)$ and we are done. \hfill{$\square$}
	
	\subsection{Computation of $\cH\lind{r}^{2} (E^{p,q} \otimes \cL\otimes B)$ via Prolongation Bundles}
	We compute the sheaf $\cH\lind{r}^{2} (E^{p,q} \otimes \cL \otimes B)$ explicitly and describe them using prolongation bundles.
	\begin{prop} \label{prop-pronlongation-computation}
		We have the following description of the sheaves $\cH\lind{r}^{2}(E^{p,q} \otimes \cL \otimes B)$ in terms of prolongation bundles:
		$$\cH\lind{r}^{2}(E^{p,q} \otimes \cL \otimes B) = \Omega_{\overline{X}}^{r} (\log D) \otimes \cP_{\alpha} E^{p,q} \otimes \cL \otimes B.$$
	\end{prop}
	{\it Proof.} These two sheaves agree on the open locus $X$ since they are both extensions of $\Omega_{X}^{r} \otimes E^{p,q} \otimes \cL\otimes B$. Hence, it is enough to examine what happens on the boundary. Fix an admissible coordinate $(\Omega; z_{1},\ldots, z_{n})$. Assume that $D \cap \Omega$ is defined by $z_{1}\cdots z_{l} = 0$ and assume for simplicity that $D_{i} = (z_{i} = 0)$. Let $\sigma$ be a nonvanishing section on $\cL\otimes B$ and let $u = \sum u_{J} \otimes \sigma dz_{J}$ for $u_{J} \in \Gamma(\Omega\sta , E^{p,q})$. Since $\omega$ and $\omega_{\Poin}$ are mutually bounded and $e^{-\xi}$ and $\prod_{j=1}^{l} |z_{j}|^{2a_{j}} (-\log |z_{j}|^{2})^{b_{j}}$ are mutually bounded on $\Omega\sta$, $u \in \cH\lind{r}^{2}(E^{p,q} , \cL \otimes B)$ if and only if
	\begin{equation} \label{eqn-integrability}
		\int_{\Omega\sta} |u_{J}|_{h_{E}}^{2} \prod_{\substack{1 \leq j \leq l \\ j \notin J} } \frac{1}{|z_{j}|^{2} (-\log |z_{j}|^{2})^{2}} \cdot \prod_{j=1}^{l} |z_{j}|^{2a_{j}} ( -\log |z_{j}|^{2})^{b_{j}} d\mu < + \infty 
	\end{equation}
	for all $J$. Here $d\mu$ is the usual Lebesgue measure on $\Omega$. Suppose $$u \in \Gamma(\Omega\sta, \Omega_{\overline{X}}^{r} (\log D) \otimes \cP_{\alpha}E^{p,q} \otimes \cL \otimes B).$$ Then we can express $u$ as
	$$ u = \sum_{J}\left( \prod_{ \substack{1 \leq j \leq l \\ j \in J}} \frac{1}{z_{j}} \right) v_{J} \otimes \sigma dz_{J} $$
	for $v_{J} \in \cP_{\alpha}E^{p,q}$. In particular, $|v_{J}|_{h_{E}} \lesssim \prod_{j=1}^{l} |z_{j}|^{- (\alpha_{j} + a_{j})/ 2} $ since $a_{j} > \alpha_{j}$. Then
	\begin{align*}
		& \int_{\Omega\sta} |v_{J}|_{h_{E}}^{2} \cdot \prod_{\substack{1 \leq j \leq l \\ j \in J}} \frac{1}{|z_{j}|^{2}} \prod_{\substack{1 \leq j \leq l \\ j \notin J} } \frac{1}{|z_{j}|^{2} (-\log |z_{j}|^{2})^{2}} \cdot \prod_{j=1}^{l} |z_{j}|^{2a_{j}} ( -\log |z_{j}|^{2})^{b_{j}} d\mu \\
		& \leq \int_{\Omega\sta} |v_{J}|_{h_{E}}^{2} \prod_{1 \leq j \leq l} |z_{j}|^{2a_{j} - 2}  \cdot \prod_{j=1}^{l} (-\log |z_{j}|^{2})^{C + 3} d\mu\\
		& \lesssim \int_{\Omega\sta} \prod_{j=1}^{l} |z_{j}|^{a_{j} - \alpha_{j} - 2} (-\log |z_{j}|^{2})^{C + 3} d\mu < + \infty
	\end{align*}
	since $a_{j} > \alpha_{j}$. This implies $\Omega_{\overline{X}}^{r}(\log D) \otimes \cP_{\alpha}E^{p,q} \otimes \cL \otimes B\subset \cH\lind{r}^{2}(E^{p,q} \otimes \cL \otimes B)$. Similarly, we can show that $\cH\lind{r}^{2}(E^{p,q}\otimes \cL \otimes B) \subset \Omega_{\overline{X}}^{r} (\log D) \otimes \cP_{a}E^{p,q} \otimes\cL \otimes B$ since if the integral \eqref{eqn-integrability} is finite, then we have $|u_{J}|_{h_{E}}^{2} \lesssim \prod_{j = 1}^{l} |z_{j}|^{-2a_{j}} \cdot \prod_{\substack{1 \leq j \leq l\\ j \in J}} |z_{j}|^{-2}$. However, we chose $a_{j}$ to be sufficiently close to $\alpha_{j}$ so that $\cP_{\alpha}E^{p,q} = \cP_{a}E^{p,q}$. Hence, the equality holds. \hfill{$\square$}
	
	\subsection{Proof of the Theorem \ref{theo-Main}}
	We finally present the proof of the main theorem. First, notice that the global bound for the Higgs field $\theta$ in Theorem \ref{theo-HiggsBound} gives a morphism on sheaves $L\lind{r, s}^{2}(E^{p,q} \otimes \cL \otimes B) \xrightarrow{\theta} L\lind{r+1, s}^{2} (E^{p,q} \otimes \cL\otimes B)$. Therefore, analogously to Section \ref{sec-deRham-Dolbeault}, we can construct a double complex:
	$$ \begin{tikzcd}
		L\lind{0,n}^{2} (E^{p,q} \otimes \cL\otimes B) \ar[r, "\theta"]  & L\lind{1,n}^{2} (E^{p-1,q+1} \otimes \cL\otimes B) \ar[r, "\theta"] & \cdots \ar[r, "\theta"]  & L\lind{n,n}^{2} (E^{p-n,q+n} \otimes \cL\otimes B)\\
		\vdots \ar[r, "\theta"] \ar[u, "\db"] & \vdots \ar[r, "\theta"]\ar[u, "\db"] & \vdots \ar[r, "\theta"] \ar[u, "\db"] & \vdots \ar[u, "\db"] &\\
		L\lind{0,1}^{2} (E^{p,q} \otimes \cL\otimes B) \ar[r, "\theta"] \ar[u, "\db"] & L\lind{1,1}^{2} (E^{p-1,q+1} \otimes \cL\otimes B) \ar[r, "\theta"] \ar[u, "\db"] & \cdots \ar[r, "\theta"] \ar[u, "\db"] & L\lind{1,n}^{2} (E^{p-n,q+n} \otimes \cL\otimes B) \ar[u, "\db"] \\
		L\lind{0,0}^{2} (E^{p,q} \otimes \cL\otimes B) \ar[r, "\theta"] \ar[u, "\db"] & L\lind{1,0}^{2} (E^{p-1,q+1} \otimes \cL\otimes B) \ar[r, "\theta"] \ar[u, "\db"] & \cdots \ar[r, "\theta"] \ar[u, "\db"] & L\lind{n,0}^{2} (E^{p-n,q+n} \otimes \cL\otimes B) \ar[u, "\db"]. 
	\end{tikzcd}$$
	By Proposition \ref{prop-L2-deRham} and \ref{prop-pronlongation-computation}, the $r$-th column is a resolution of $\Omega_{\overline{X}}^{r}(\log D) \otimes E_{\alpha}^{p-r,q+r} \otimes \cL \otimes B$ by fine sheaves. Hence, we can compute the hypercohomology of
	$$ \Big[ E_{\alpha}^{p,q} \to \Omega_{\overline{X}}^{1}(\log D) \otimes E_{\alpha}^{p-1, q+1} \to \cdots \to \Omega_{\overline{X}}^{n}(\log D) \otimes E_{\alpha}^{p-n, q+n} \Big] [n] \otimes \cL \otimes B$$
	by taking the global section of the total complex above and compute the cohomology. Let $(\bb{E}, \eth)$ be the total complex of the double complex above. We have a concrete description of the global section of the sheaves $L\lind{r,s}^{2}(E^{p-r, q+r} \otimes \cL \otimes B)$. The global sections $u$ are $(r, s)$-forms on $X$ with values in $E^{p-r, q+r} \otimes \cL \otimes B$ with measurable coefficients such that
	$$ \int_{X} \norm{u}_{h', \omega}^{2} dV_{\omega} < + \infty \qquad \text{and} \qquad \int_{X} \norm{\db u}_{h' , \omega}^{2} dV_{\omega} < + \infty. $$
	By Proposition \ref{prop-constr-Kahler} and \ref{prop-deRham-Apriori}, we have an a priori inequality
	$$ \norm{\eth \bfu}^{2} + \norm{\eth\sta \bfu}^{2} \geq 0.1 \norm{\bfu}^{2}$$
	for $\bfu \in \bb{E}^{l}$ when $l > 0$. Since $(X,\omega)$ is complete and $|\theta|_{h_{E} , \omega}^{2}$ is globally bounded on $X$, the vanishing of cohomology immediately follows from Proposition \ref{prop-L2-global}. \hfill{$\square$}

	\section{Proof of Saito's vanishing theorem}
	In this section, we prove Saito's vanishing theorem by reducing it to Theorem \ref{theo-Main-ample}. We first give a proof for Saito's Hodge modules, and then explain how the same method works in the setting of complex Hodge modules. For a projective variety $X$, we denote by $\MHM(X)$ the abelian category of graded polarizable mixed Hodge modules on $X$, in the sense of Saito. Roughly speaking, a mixed Hodge module is a filtered $\cD$-module with extra structure. The Saito vanishing theorem concerns the graded piece of the de Rham complex of mixed Hodge modules.
	
	\begin{theo}[\cite{saito1990mixed}] \label{theo-Saito-vanishing-general}
		Let $X$ be a complex projective variety and let $\cM \in \MHM(X)$ be a graded polarizable mixed Hodge module on $X$. For every ample line bundle $\cL$ on $X$, we have
		$$ \HH^{l}\big( X, \gr_{p} \DR_{X}(\cM) \otimes \cL \big) = 0 \qquad \textnormal{for all } l > 0, p \in \ZZ.$$
	\end{theo}
	
	We first review some necessary facts on Hodge modules and then prove the Saito vanishing at the end.
	
	\subsection{Terminology and Properties of Hodge Modules}
	
	We give a very brief summary on Hodge modules. All of the results can be found in two foundational papers \cite{saito1988modules,saito1990mixed}, or in a nice survey \cite{schnell2014overview}. Let $X$ be a smooth complex algebraic variety. Saito defines two categories $\HM(X, w)$ and $\MHM(X)$ which are the category of pure Hodge modules with weight $w$ and the category of graded polarizable mixed Hodge modules. The objects in $\HM(X, w)$ are triples $((\cM, F_{\bullet}), K, \alpha)$ where $(\cM, F_{\bullet})$ is a holonomic filtered $\cD$-module, $K$ is a $\QQ$-perverse sheaf on $X$, and an isomorphism $\alpha : \DR_{X}^{\mathrm{an}} \cM \xrightarrow{\simeq} K \otimes_{\QQ} \CC$. For mixed Hodge modules, we have an extra \textit{weight} filtration on $K$, hence on $\cM$. The objects in this category are required to satisfy suitable compatibility conditions. If $\cN$ is a variation of Hodge structures on $X$ of weight $w$ and we switch from decreasing filtration to increasing filtration by the rule $F_{k}\cN = F^{-k}\cN$, then $\cN$ underlies a structure of a pure Hodge module of weight $w + \dim X$. As opposed to the previous sections, we switch to increasing filtrations for Hodge modules from this point in order to be consistent with the conventions on $\cD$-modules. We introduce some terminology and properties of $\cD$-modules and Hodge modules, without giving any proofs. In this article, a $\cD$-module is understood to mean a left $\cD$-module.
	\begin{enumerate} 
		\item For a singular variety $X$, we embed $X$ in a smooth variety $Z$ and define $\MHM(X)$ as the category of mixed Hodge modules on $Z$ whose support is in $X$. The category $\MHM(X)$ does not depend on the choice of closed embedding $X \hookrightarrow Z$. We can similarly define $\HM(X, w)$ as well. We can still define these categories even if $X$ cannot be embedded into a smooth variety by covering with embeddable open subsets and imposing some conditions on the intersections. However, we ignore this issue since we will work only on (quasi)-projective varieties.
		
		\item The filtration on the de Rham complex of a mixed Hodge module $\cM$ on a smooth variety $X$ is given by
		$$ F_{k} \DR_{X}(\cM) = \left[F_{k}\cM \to \Omega_{X}^{1} \otimes F_{k+1}\cM \to \cdots \to \Omega_{X}^{n} \otimes F_{k+n}\cM\right] [n].$$
		After taking the graded pieces, we obtain a complex of $\cO_{X}$-modules
		$$ \gr_{k}\DR_{X}(\cM) = \left[ \gr_{k} \cM \to \Omega_{X}^{1} \otimes \gr_{k+1} \cM \to \cdots \to\Omega_{X}^{n} \otimes \gr_{k+n} \cM\right][n]. $$
		We can extend the notion of graded de Rham complexes to objects in the derived category of mixed Hodge modules. In other words, for $u \in \cD^{b}(\MHM(X))$, the graded de Rham complex $\gr_{k} \DR_{X}(u)$ makes sense as an object in $\cD_{\coh}^{b}(\cO_{X})$.
		
		\item  \cite{schnell2016saito}*{Lemma 7.3}. For a singular variety $X$, we define the graded de Rham complex of $\cM \in \MHM(X)$ by embedding $X$ in a smooth variety $Z$ and taking the graded de Rham complex in $Z$. It turns out that the complex $\gr_{k}\DR_{X}(\cM)$ is actually a complex of $\cO_{X}$-modules, and does not depend on the choice of the embedding $X \hookrightarrow Z$ as an object in $\cD_{\coh}^{b}(\cO_{X})$.
		
		\item For holonomic $\cD$-modules there are functors $f\lsta, f_{!}, \DD_{X}$ defined at the level of derived category of holonomic $\cD$-modules. The same functors admit enrichment to functors between the corresponding derived categories of mixed Hodge modules.
		
		\item We have a canonical transformation $\gamma_{f} : f_{!} \to f\lsta$ and this transformation is an isomorphism if $f$ is proper.
		
		\item If $j : U \to X$ is a locally closed immersion between smooth varieties and $\cM \in \MHM(U)$, then the transformation $j_{!}\cM \to j\lsta \cM$ induces a map $\cH^{0}(j_{!} \cM) \to \cH^{0}(j\lsta \cM)$. The image of this map is the intermediate extension of $\cM$ which is denoted by $j_{!\ast} \cM$.
		
		\item Furthermore, let's suppose that $U \subset X$ is open and $D = X \setminus U$ is a divisor. Then $j\lsta$ and $j_{!}$ are induced by an exact functor at the level of abelian categories. In other words, if $\cM \in \MHM(U)$ is a mixed Hodge module on $U$, then we can view $\cM[\ast D] := j\lsta \cM$ and $\cM[!D]:= j_{!} \cM$ as objects inside $\MHM(X)$. Moreover, the underlying $\cO_{X}$-module structure of $\cM[\ast D]$ agrees with the usual sheaf theoretic push-forward.
		
		\item For proper morphisms $\mu \colon Y \to X$, taking the graded de Rham complex commutes with taking the pushforward. In other words, if $\cM$ is a mixed Hodge module on $Y$, we have
		$$ \gr_{p}\DR_{X} (\mu\lsta \cM) \simeq \bfR \mu\lsta( \gr_{p}\DR_{Y} \cM). $$
	\end{enumerate}
	
	We next describe the structure of the categories $\MHM(X)$ and $\HM(X, w)$.
	
	\begin{enumerate}	
		\item For a mixed Hodge module $\cM \in \MHM(X)$, the graded piece $\gr_{w}^{W}(\cM) \in \HM(X, w)$ is a pure Hodge module of weight $w$.
		
		\item We say that a pure Hodge module $\cM \in\HM(X, w)$ has strict support $Z$ if the support of every nonzero submodule and quotient module is equal to $Z$. We denote by $\HM_{Z}(X, w)$ the subcategory of $\HM(X, w)$ consisting of objects with strict support $Z$.
		
		\item Every pure Hodge module $\cM \in \HM(X, w)$ admits a decomposition by strict support. Hence $\cM$ is of the form
		$$ \cM = \bigoplus_{Z \subset X} \cM_{Z}$$
		where the sum runs over all irreducible closed subsets $Z \subset X$ and $\cM_{Z}$ is a pure Hodge module with strict support $Z$.
		
		\item Every object $\cM \in \HM_{X}(X, w)$ is a generically defined variation of Hodge structures of weight $w - \dim X$. Conversely, every variation of Hodge structures $\cN$ of weight $w - \dim X$ defined on a Zariski open subset of the smooth locus of $X$ extends uniquely to an object in $\HM_{X}(X, w)$ via the intermediate extension.
	\end{enumerate}
	
	\subsection{Logarithmic comparison}
	Let $X$ be a smooth variety and let $D$ be an SNC divisor on $X$. Let $U$ be the complement $X \setminus D$ and denote the open inclusion as $j \colon U \to X$. We consider a variation of Hodge structures $E$ on $U$. The main theorem involves the logarithmic de Rham complex of $E$. We now describe the relation between the logarithmic de Rham complex $\DR_{(X, D)}(E_{\alpha})$ and the de Rham complex of the $\cD_{X}$-modules $E[\ast D]$ and $E[!D]$. The results can be found in \cite{saito1990mixed}*{3.b} or \cite{wu2017multi}. We define a subring of $\cD_{X}$ by
	$$ V_{0}^{D} \cD_{X} = \{ P \in \cD_{X} : P \cdot (f)^{j} \subset (f)^{j} \text{ for all } j \}$$
	where $D$ is given by $f = 0$. If $(z_{1},\ldots, z_{n})$ is a system of coordinates such that $D$ is given by the equation $z_{1}\cdots z_{l} = 0$, then we can describe $V_{0}^{D} \cD_{X}$ locally as
	$$V_{0}^{D}\cD_{X} = \cO_{X} \langle z_{1} \de_{1},\ldots, z_{l} \de_{l}, \de_{l+1} , \ldots  , \de_{n} \rangle. $$
	 Giving a $V_{0}^{D} \cD_{X}$-module structure on $\cN$ is equivalent to giving a logarithmic connection
	$$ \nabla \colon \cN \to \Omega_{X}^{1} (\log D) \otimes \cN $$
	such that $\nabla^{2} = 0$. We can define the logarithmic de Rham complex of a $V_{0}^{D} \cD_{X}$-module as follows:
	$$ \DR_{(X, D)} \cN = \left[\cN \to \Omega_{X}^{1} (\log D) \otimes \cN \to \ldots \to \Omega_{X}^{n} (\log D) \otimes \cN\right][n].$$
	We have a filtration $F_{\bullet}$ on $V_{0}^{D} \cD_{X}$ which is induced by the order filtration on $\cD_{X}$. For a variation of Hodge structures $E$ on $U$, the Deligne extension $E_{\alpha}$ has a structure of filtered $V_{0}^{D} \cD_{X}$-module by the logarithmic connection and the Hodge filtration which is constructed via the nilpotent orbit theorem. The extensions $E[\ast D]$ and $E[!D] $ are given by
	$$ E[!D] = \cD_{X} \otimes_{V_{0}^{D} \cD_{X}} E_{1 - \eps}  \qquad \textnormal{and} \qquad  E[\ast D] = \cD_{X} \otimes_{V_{0}^{D} \cD_{X}} E_{1} \simeq \cD_{X}\cdot E_{1}. $$
	Here, we chose $\eps > 0$ to be small enough so that the eigenvalues of the residues of $E_{1-\eps}$ lie inside the interval $(-1, 0]$. The inclusion $E_{1 - \eps} \to E_{1}$ induces a morphism $E[!D] \to E[\ast D]$ and this agrees with the canonical one. The image of the natural map $ E[!D] \to E[\ast D]$ is exactly $\cD_{X} \cdot E_{1 -\eps}$, which is the intermediate extension $j_{!\ast} E$. The filtrations on $E[!D]$ and $E[\ast D]$ are defined by the filtrations on $\cD_{X}$ and $E_{\alpha}$. To be precise, we have
	\begin{align*}
		& F_{p} E[\ast D] = \sum_{i} F_{p-i} \cD_{X} \otimes F_{i} E_{1}  \\
		& F_{p} E [!D] = \sum_{i} F_{p - i} \cD_{X}\otimes F_{i} E_{ 1 - \eps} .
	\end{align*}
	Also, we can define filtered $\cD$-modules $\cD_{X}\otimes_{V_{0}^{D} \cD_{X}} E_{\alpha}$ for each $\alpha \in \RR^{l}$. The obvious map
	$$ E_{\alpha} \to \cD_{X} \otimes_{V_{0}^{D}\cD_{X}} E_{\alpha+ 1} $$
	obtained by the inclusion $E_{\alpha} \hookrightarrow E_{\alpha + 1}$ composed with $m \mapsto 1 \otimes m$ gives a morphism at the level of de Rham complexes.
	$$ \DR_{(X, D)} E_{\alpha} \to \DR_{X} \left(\cD_{X} \otimes_{V_{0}^{D}\cD_{X}} E_{\alpha+ 1}\right).$$
	The upshot is that we can compute the de Rham complex on the right hand side using the logarithmic de Rham complex.
	
	\begin{theo}[\cite{saito1990mixed}*{3.11}, \cite{wu2017multi}*{Theorem 5.3.8}]
		For each $\alpha \in \RR^{l}$, the morphism 
		$$ \DR_{(X, D)} E_{\alpha} \to \DR_{X}\left(\cD_{X} \otimes_{V_{0}^{D}\cD_{X}} E_{\alpha + 1}  \right)$$
		is a filtered quasi-isomorphism.
	\end{theo}

	If we apply this theorem for $\alpha = 0$ and $\alpha = (-\eps, \ldots, -\eps) \in \RR^{l}$ and take the graded pieces, we get the following isomorphisms:
	
	\begin{prop} \label{prop-log-comparison}
		With the above notation, we get an isomorphism between the graded pieces of the de Rham complexes:
		\begin{align*}
			& \gr_{p} \DR_{(X, D)} E_{0} \simeq_{qis} \gr_{p} \DR_{X} (E[\ast D]) \\
			& \gr_{p} \DR_{(X, D)} E_{-\eps} \simeq_{qis} \gr_{p} \DR_{X} (E[!D]).
		\end{align*}
	\end{prop}
	
	This allows us to compute the de Rham complexes of $E[\ast D]$ and $E[!D]$ in terms of the logarithmic de Rham complexes.
	
	\begin{rema}
		Both \cite{saito1990mixed} and \cite{wu2017multi} assume quasi-unipotence of the monodromy to apply the nilpotent orbit theorem. However, this is no longer an issue since we have the nilpotent orbit theorem for arbitrary complex polarized Hodge structures by \cite{sabbah2022degenerating} and \cite{deng2022nilpotent}. Then the same proof goes through with a slight change that the Deligne extensions $E_{\alpha}$ might no longer be $\QQ$-indexed.
	\end{rema}

	\subsection{Proof of Saito's vanishing theorem} \label{sec:proof-Saito-Q}
	We first reduce the general case to that of pure Hodge modules with strict support. Every mixed Hodge module $\cM$ has a weight filtration $W_{\bullet} \cM$ such that $\gr_{w}^{W} \cM$ is a pure Hodge module of weight $w$. Using the exactness of the functor $\gr_{p} \DR_{X}$, we can easily deduce the Saito vanishing theorem from vanishing theorems for pure Hodge modules. Furthermore, since any pure Hodge module admits a decomposition by strict support, the assertion in Theorem \ref{theo-Saito-vanishing-general} is a consequence of the following.
	
	\begin{theo} \label{theo-Saito-vanishing-purestrict}
		Let $X$ be a reduced and irreducible projective variety and let $\cM \in \HM_{X}(X, w)$ be a polarizable Hodge module with strict support $X$. For an ample line bundle $\cL$ on $X$, we have
		$$ \HH^{l}(X, \gr_{p} \DR_{X}(\cM) \otimes \cL ) = 0 \qquad \textnormal{for all } l > 0, p \in \ZZ.$$
	\end{theo}
	We present a key step which relies on our main theorem. Namely, we can get vanishing for Hodge modules obtained by `MHM-theoretic' push-forward and push-forward with compact support across a possibly singular divisor. Before that, we prove a short lemma about constructing nice resolution of singularities. The argument is essentially due to \cite{kollar2021resolution}*{Lemma 8}.

    \begin{lemm} \label{lemm-log-res}
        Let $X$ be a normal variety and $D$ be an effective Cartier divisor on $X$ such that $X \setminus D$ is smooth. Then there exists a log resolution $\mu : Y \to X$ of $(X, D)$ such that
        \begin{enumerate}
            \item $\mu : Y \to X$ is an isomorphism over $X \setminus D$, and
            \item There exists a divisor $E$ on $Y$ such that $E$ is $\mu$-ample and $\mu(\supp E) \subset D$.
        \end{enumerate}
    \end{lemm}
    \begin{proof}
        Pick a log resolution $\pi_{1} : X_{1} \to X$ which is an isomorphism over $U= X \setminus D$. Pick a $\pi_{1}$-ample divisor $H_{1}$ on $X_{1}$ and let $\cL = \cO_{X}(-\pi_{1 \ast}H_{1})$. Let
        $$\pi_{2} : X_{2} = \SProj \left( \bigoplus_{m \geq 0} \cL^{\otimes m} \right) \to X$$
        and denote $H_{2} = \cO_{X_{2}}(1)$. It is clear that $\pi_{2}$ is an isomorphism over $U$ since $\pi_{1\ast} H_{1}$ is Cartier on $U$. Moreover, $H_{2}$ is $\pi_{2}$-ample. We choose a log resolution $\pi_{3} : X_{3} \to X$ which is an isomorphism over $U$ that dominates $X_{2}$ and $X_{1}$. Hence, we have the following diagram
        $$ \begin{tikzcd}
            & X_{3} \ar[rd, "\tau_{1}"] \ar[ld, "\tau_{2}"'] \ar[dd, "\pi_{3}"] & \\ X_{2} \ar[rd, "\pi_{2}"'] & & X_{1} \ar[ld, "\pi_{1}"] \\
            & X & 
        \end{tikzcd}$$
        Since $X_{1}$ is smooth, there is a $\tau_{1}$-exceptional divisor $F$ which is $\tau_{1}$-ample. For $m \gg 0$, the divisor $m\tau_{1}\sta H_{1} + F$ is $\pi_{3}$-ample. Since $H_{2}$ is $\pi_{2}$-ample, $\tau_{2}\sta H_{2}$ is $\pi_{3}$-nef. Therefore
        $$ m \tau_{1}\sta H_{1} + m \tau_{2}\sta H_{2} + F$$
        is $\pi_{3}$-ample. However, $\cO_{X_{3}}(m\tau_{1}\sta H_{1} + m \tau_{2}\sta H_{2} + F)|_{U} \simeq \cO_{U}$ by construction. This means that there is a rational section of $m\tau_{1}\sta H + m\tau_{2}\sta H_{2} + F $ whose divisor $E$ is supported in $X_{3} \setminus U$. The morphism $\pi_{3} : X_{3} \to X$ and $E$ is what we wanted.
    \end{proof}

	\begin{prop} \label{prop-vanishing-singulardivisor}
		Let $X$ be a projective variety and $D$ an effective Cartier divisor on $X$. Suppose that $U = X \setminus D$ is smooth and consider the open inclusion $j \colon U \to X$. Let $\cM$ be a variation of Hodge structures on $U$ that we identify with the corresponding pure Hodge module. If $\cM[\ast D] := j\lsta \cM$ and $\cM[!D] := j_{!} \cM$, then the following vanishings hold:
		\begin{align*}
			&\HH^{l}(X, \gr_{p}\DR_{X}(\cM[\ast D]) \otimes \cL) = 0 \\
			&\HH^{l}(X, \gr_{p}\DR_{X}(\cM[! D]) \otimes \cL) = 0 
		\end{align*}
		for all $l > 0$, $p \in \ZZ$ and every ample line bundle $\cL$ on $X$.
	\end{prop}
	
	{\it Proof.} 
        First, we reduce to the case when $X$ is normal. Consider the normalization $\tau : X^{\mathrm{nor}} \to X$ which is clearly an isomorphism over $U$ and let $F = (\tau\sta D)_{\mathrm{red}}$. Therefore, we have inclusions $i : U \to X^{\mathrm{nor}}$ and $j : U \to X$. Since $\tau$ is finite and hence proper, we have the following isomorphisms:
        \begin{align*}
		& \gr_{p} \DR_{X}(\cM[\ast D]) \simeq  \tau\lsta \left( \gr_{p} \DR_{X^{\mathrm{nor}}}(\cM[\ast F]) \right)  \\
		& \gr_{p} \DR_{X}(\cM[!D]) \simeq  \tau\lsta \left( \gr_{p} \DR_{X^{\mathrm{nor}}} (\cM[!F]) \right) .
	\end{align*}
        The isomorphism follows by the commutativity of the graded de Rham complex with proper pushforward. By the projection formula, it is enough to show the two following vanishings:
        \begin{align*}
            & \HH^{l}(X^{\mathrm{nor}}, \gr_{p} \DR_{X^{\mathrm{nor}}}(\cM[\ast F])\otimes \tau\sta\cL) = 0 \qquad \text{for all } l > 0, p \in \ZZ, \\
            & \HH^{l}(X^{\mathrm{nor}}, \gr_{p} \DR_{X^{\mathrm{nor}}}(\cM[! F])\otimes \tau\sta\cL) = 0 \qquad \text{for all } l > 0, p \in \ZZ.
        \end{align*}
        Since $\tau\sta \cL$ is still ample, it is enough to show the assertion when $X$ is normal.
 
	Now, we assume that $X$ is normal. Consider a log resolution $\mu \colon Y \to X$ of the pair $(X, D)$ which is an isomorphism over $U$. Let $F = (\mu\sta D)_{\mathrm{red}}$, which is an SNC divisor on $Y$. Write $F = \sum_{i=1}^{s} F_{i}$. By Lemma \ref{lemm-log-res}, we can assume that there exists $c_{i}$ such that $\sum_{i=1}^{s} c_{i} F_{i}$ is $\mu$-ample. Since $\mu$ is proper, we can compute the graded de Rham complex of $\cM[\ast D]$ and $\cM[!D]$ as follows:
	\begin{align*}
		& \gr_{p} \DR_{X}(\cM[\ast D]) \simeq \bfR \mu\lsta \left( \gr_{p} \DR_{Y}(\cM[\ast F]) \right) \simeq \bfR \mu\lsta \left( \gr_{p} \DR_{(Y, F)}(\cM_{0}) \right) \\
		& \gr_{p} \DR_{X}(\cM[!D]) \simeq \bfR \mu\lsta \left( \gr_{p} \DR_{Y} (\cM[!F]) \right) \simeq \bfR \mu\lsta \left( \gr_{p} \DR_{(Y, F)} (\cM_{-\eps})\right).
	\end{align*}
	The isomorphism on the right hand side follows by logarithmic comparison (Proposition \ref{prop-log-comparison}). By the projection formula, it is enough to show the two following vanishings:
	\begin{align*}
		& \HH^{l}(Y, \gr_{p} \DR_{(Y, F)}(\cM_{0}) \otimes \mu\sta \cL) = 0 \qquad \text{for all } l > 0, p \in \ZZ, \\		
		& \HH^{l}(Y, \gr_{p} \DR_{(Y, F)}(\cM_{-\eps}) \otimes \mu\sta \cL) = 0 \qquad \text{for all } l > 0, p\in \ZZ.
	\end{align*}
 	The only problem is that $\mu\sta \cL$ is no longer ample. However, we can handle this by perturbing $\mu\sta \cL$. Note that for $0 < d \ll c \ll 1$, we have $\mu\sta (\cL + cD) + d \cdot \sum_{i=1}^{s} c_{i}F_{i}$ ample. Hence, there exist $0 < a_{i} \ll 1$ such that $\mu\sta \cL + \sum_{i=1}^{s} a_{i} F_{i}$ is ample since the images of $F_{i}$ are contained in $D$. In fact, we can choose $a_{i}$ to be sufficiently small so that $\cP_{0}\cM^{p,q} = \cP_{a}\cM^{p,q}$ as a sheaf on $Y$. Then we can use Theorem \ref{theo-Main-ample} to get
	$$ \HH^{l}(Y, \gr_{p} \DR_{(Y, F)}(\cM_{0}) \otimes \mu\sta \cL) = \HH^{l}(Y, \gr_{p} \DR_{(Y, F)}(\cM_{a}) \otimes \mu\sta \cL) = 0 \qquad \text{for all } l >0. $$
	The second vanishing follows along the same lines. There exist $0 < d \ll c \ll 1$ such that $\mu\sta (\cL - cD) + d\cdot \sum_{i=1}^{s} c_{i} F_{i}$ is ample. Hence, there exist $0< a_{i} \ll 1$ such that $\mu\sta \cL- \sum_{i=1}^{s} a_{i} F_{i}$ is ample. We can make these $a_{i}$ be sufficiently small so that $0 < a_{i} < \eps$, which implies $\cP_{-\eps} \cM^{p,q} = \cP_{-a} \cM^{p,q}$ as a sheaf on $Y$. Therefore,
	$$ \HH^{l}(Y, \gr_{p} \DR_{(Y, F)}(\cM_{ -\eps}) \otimes \mu\sta \cL) = \HH^{l}(Y, \gr_{p} \DR_{(Y, F)}(\cM_{ -a}) \otimes \mu\sta \cL) = 0 \qquad \text{for all } l >0. $$
	This concludes the proof. \hfill{$\square$}
	
	We finally give the proof of the Saito vanishing theorem.
	
	{\it Proof of Theorem \ref{theo-Saito-vanishing-general}.} The strategy is to prove Theorems \ref{theo-Saito-vanishing-general} and \ref{theo-Saito-vanishing-purestrict} simultaneously by Noetherian induction. As we have already seen, Theorem \ref{theo-Saito-vanishing-purestrict} in dimension $\leq n$ implies Theorem \ref{theo-Saito-vanishing-general} in dimension $\leq n$. Let $X$ be an irreducible complex projective variety and let $\cM$ be a pure Hodge module with strict support $X$. Then there exists an open subset $U \subset X$ such that $\cN = \cM|_{U}$ is a variation of Hodge structures on $U$. By shrinking $U$, we may assume that $D = X \setminus U$ is a divisor. Fix a very ample divisor $H$ such that $|H - D| \neq \emptyset$ and write $H \sim D + E$. By replacing $U$ by $X \setminus (D \cup E)$ and $D$ by $D \cup E$, we may assume that $D$ is a hyperplane section of $X$ for some embedding $X \hookrightarrow \PP^{N}$. Considering every $\cD$-module as a $\cD_{\PP^{N}}$-module, we see that $\cN[!D]$ is an object in $\MHM(X)$ because $\PP^{N} \setminus H \to \PP^{N}$ is an affine morphism. Since $\cM$ is the intermediate extension of $\cN$, we have an exact sequence
	$$ 0 \to \cK \to \cN[!D] \to \cM \to 0$$
	and the kernel $\cK$ is supported inside $D$. By Proposition \ref{prop-vanishing-singulardivisor} and Noetherian induction, we have
	\begin{align*}
		& \HH^{l}(X, \gr_{p} \DR_{X}(\cN[!D]) \otimes \cL) = 0 \qquad \text{for all } l > 0, p \in \ZZ,\\
		& \HH^{l}(X, \gr_{p} \DR_{X}(\cK) \otimes \cL) = 0 \qquad \text{for all } l > 0, p \in \ZZ.
	\end{align*}
	The long exact sequence and the exactness of $\gr_{p} \DR_{X}$ gives
	$$ \HH^{l}(X, \gr_{p} \DR_{X}(\cM) \otimes \cL) = 0 \qquad \text{for all } l > 0, p \in \ZZ.$$
	This proves the general version of Saito's vanishing theorem. \hfill{$\square$}

    \subsection{Vanishing theorem for complex Hodge modules} We finally explain how to modify the previous proof of the Saito vanishing theorem in the case of \textit{complex} Hodge modules. Before that, we briefly highlight the key differences between Saito's Hodge modules and the Hodge modules of Sabbah--Schnell. In the theory of Sabbah--Schnell, they remove the underlying perverse sheaf $K$ and the isomorphism $\alpha \colon \DR_{X}^{\mathrm{an}}\cM \simeq K \otimes_{\QQ} \CC$ from the picture by considering a sesquilinear pairing on $\cD$-modules. More precisely, instead of considering a underlying perverse sheaf, they consider a \textit{triple} $(\cM, \cM', \mathfrak{s})$ where $\cM$ and $\cM'$ are filtered holonomic $\cD$-modules, and
    $$ \mathfrak{s} \colon \cM \otimes \overline{\cM'} \to \mathfrak{Db}_{X}$$
    is a $\cD_{X} \otimes \cD_{\overline{X}}$-linear perfect pairing (which gives the data of the polarization). Here, $\mathfrak{Db}_{X}$ is the sheaf of complex (0,0)-currents on $X$, hence carries a natural left action of $\cD_{X} \otimes \cD_{\overline{X}}$.

    In \cite{SSMHMproject}, for a complex manifold $X$, the authors construct the category $\mathrm{pHM}_{\CC}(X, w)$\footnote{The subscript $\CC$ is not written in the MHM project, but we write it in order to distinguish it with Saito's Hodge modules.} of complex pure polarizable Hodge modules on $X$ with weight $w$. At the time of the writing, \cite{SSMHMproject} does not contain the definition of complex \textit{mixed} Hodge modules. We use \cite{davis2022mixed}*{\S7.1} for the construction of the category $\MHM_{\CC}(X)$ of complex mixed Hodge modules, which follows the same strategy of \cite{saito1990mixed}*{\S4}. They define the category $\MHM_{\CC}(X)$ as the full-subcategory of triples which are stable under certain operations. Then, the same six functor formalism works for complex mixed Hodge modules. We point out that even if we want the vanishing theorem only for pure complex Hodge modules, the strategy in \S\ref{sec:proof-Saito-Q} needs the formalism of complex mixed Hodge modules since for a pure Hodge module $\cN$ on a complement of a divisor $D$ of $X$, we use that the push-forwards $\cN[\ast D]$ and $\cN[!D]$ are mixed Hodge modules. The analogue of this fact for complex mixed Hodge modules is in \cite{davis2022mixed}*{\S7.1}. The strategy in \S\ref{sec:proof-Saito-Q} can be applied to the setting of complex mixed Hodge modules since
    \begin{enumerate}
        \item one can reduce to the case of pure Hodge module with strict supports,
        \item the structure theorem for pure Hodge module (see \cite{SSMHMproject}*{\S16}), and
        \item Theorem \ref{theo-Main-ample} uses complex polarizable variation of Hodge structures.
    \end{enumerate}

	\bibliographystyle{plain}
    
	\bibliography{Reference}

\end{document}